\documentclass[a4paper,oneside,11pt]{article}

\usepackage[a4paper]{geometry}
\usepackage{aeguill}                 
\usepackage{graphicx}
\usepackage{amsmath,amsfonts,amssymb,amsthm}
\usepackage{comment} 
\usepackage{color}
\usepackage{epstopdf}
\usepackage{srcltx}
\usepackage[utf8]{inputenc} 
\usepackage[T1]{fontenc} 
 \usepackage{hyperref}
\usepackage[english]{babel}

\title{Boundaries of relative factor graphs and subgroup classification for automorphisms of free products}
\author{Vincent Guirardel and Camille Horbez}

\usepackage[all]{xy}
\usepackage{bbm}
\begin{document}
\maketitle
\newtheorem{de}{Definition} [section]
\newtheorem{theo}[de]{Theorem} 
\newtheorem{thmbis}{Theorem} 
\newtheorem{prop}[de]{Proposition}
\newtheorem{lemma}[de]{Lemma}
\newtheorem{cor}[de]{Corollary}
\newtheorem{propd}[de]{Proposition-Definition}

\theoremstyle{remark}
\newtheorem{rk}[de]{Remark}
\newtheorem*{rk*}{Remark}
\newtheorem{ex}[de]{Example}
\newtheorem{question}[de]{Question}

\normalsize

\newcommand{\coucou}[1]{\footnote{#1}\marginpar{$\leftarrow$}}
\newcommand{\Ccomment}[1]{\Cmodif\marginpar{\tiny\begin{center}\textcolor{blue}{#1}\end{center}}} 
\newcommand{\Ccom}[1]{\Cmodif\marginpar{\tiny\begin{center}\textcolor{blue}{#1}\end{center}}} 
\newcommand{\Vcomment}[1]{\Vmodif\marginpar{\tiny\begin{center}\textcolor{red}{#1}\end{center}}}
\newcommand{\Vmodif}{$\textcolor{red}{\clubsuit}$}
\newcommand{\Cmodif}{$\textcolor{blue}{\spadesuit}$}
\newcommand{\Cmod}{$\textcolor{blue}{\spadesuit}$}

\addtolength\topmargin{-.5in}
\addtolength\textheight{1.in}
\addtolength\oddsidemargin{-.045\textwidth}
\addtolength\textwidth{.09\textwidth}

\newcommand{\Stab}{\mathrm{Stab}}
\newcommand{\Inn}{\mathrm{Inn}}
\newcommand{\Fix}{\mathrm{Fix}}
\newcommand{\Aut}{\mathrm{Aut}}
\newcommand{\Out}{\mathrm{Out}} 
\newcommand{\ie}{i.~e. }
\newcommand{\imp}{\Rightarrow}
\newcommand{\ra}{\rightarrow}
\newcommand{\m}{^{-1}}
\newcommand{\dunion}{\sqcup}
\newcommand{\eps}{\varepsilon}
\renewcommand{\epsilon}{\varepsilon}
\newcommand{\calf}{\mathcal{F}}
\newcommand{\cala}{\mathcal{A}}
\newcommand{\cali}{\mathcal{I}}
\newcommand{\caly}{\mathcal{Y}}
\newcommand{\calx}{\mathcal{X}}
\newcommand{\calz}{\mathcal{Z}}
\newcommand{\calo}{\mathcal{O}}
\newcommand{\calb}{\mathcal{B}}
\newcommand{\calq}{\mathcal{Q}}
\newcommand{\calu}{\mathcal{U}}
\newcommand{\calg}{\mathcal{G}}
\newcommand{\call}{\mathcal{L}}
\newcommand{\cald}{\mathcal{D}}
\newcommand{\calj}{\mathcal{J}}
\newcommand{\calr}{\mathcal{R}}
\newcommand{\calh}{\mathcal{H}}
\newcommand{\Term}{\mathrm{Term}}
\newcommand{\ol}{\overline}
\newcommand{\AT}{\mathcal{AT}}
\newcommand{\FAT}{\mathcal{F\!AT}}
\newcommand{\Z}{\mathcal{Z}}
\newcommand{\Zmax}{{\mathcal{Z}_\mathrm{max}}}
\newcommand{\bbR}{\mathbb{R}}
\newcommand{\bbP}{\mathbb{P}}
\newcommand{\bbZ}{\mathbb{Z}}
\newcommand{\actson}{\curvearrowright}
\newcommand{\es}{\emptyset}
\newcommand{\grp}[1]{\langle #1 \rangle}
\newcommand{\ngrp}[1]{\langle\langle #1 \rangle\rangle}
\newcommand{\dg}{\dagger}
\newcommand{\xra}[1]{\xrightarrow{#1}}
\renewcommand{\Im}{\mathrm{Im}}
\newcommand{\Red}{\mathrm{Red}}
\newcommand{\ZA}{\mathcal{Z\!A}}
\newcommand{\Mod}{\mathrm{Mod}}
\newcommand{\Vol}{\mathrm{Vol}}
\newcommand{\Stabis}{\mathrm{Stab^{is}}}

\newcommand{\FF}{\mathrm{FF}}
\newcommand{\FS}{\mathrm{FS}}
\newcommand{\ZF}{{\Z\mathrm{F}}}
\newcommand{\ZmaxF}{{\Zmax \mathrm{F}}}
\newcommand{\ZS}{{\Z\mathrm{S}}}
\newcommand{\RC}{\mathrm{RC}}
\newcommand{\ZRC}{{\Z_{\mathrm{RC}}}}
\newcommand{\ZRCF}{{\ZRC\mathrm{F}}}
\newcommand{\bbN}{{\mathbb{N}}}
\newcommand{\ad}{\mathrm{ad}}

\makeatletter
\edef\@tempa#1#2{\def#1{\mathaccent\string"\noexpand\accentclass@#2 }}
\@tempa\rond{017}
\makeatother

\newcommand{\smallsim}{\smallsym{\mathrel}{\sim}}

\makeatletter
\newcommand{\smallsym}[2]{#1{\mathpalette\make@small@sym{#2}}}
\newcommand{\make@small@sym}[2]{%
  \vcenter{\hbox{$\m@th\downgrade@style#1#2$}}%
}
\newcommand{\downgrade@style}[1]{%
  \ifx#1\displaystyle\scriptstyle\else
    \ifx#1\textstyle\scriptstyle\else
      \scriptscriptstyle
  \fi\fi
}
\makeatother

\newcommand{\simAT}{{\smallsim}_{\!\AT}}
\newcommand{\simFAT}{{\smallsim}_{\!\FAT}}
\newcommand{\simZA}{{\smallsim}_{\mkern-1mu \ZA}}

\begin{abstract}
Given a countable group $G$ splitting as a free product $G=G_1\ast\dots\ast G_k\ast F_N$, we establish classification results for subgroups of the group $\text{Out}(G,\calf)$ of all outer automorphisms of $G$ that preserve the conjugacy class of each $G_i$.  
We show that every finitely generated subgroup $H\subseteq\text{Out}(G,\calf)$ either contains a relatively fully irreducible automorphism, or else it virtually preserves the conjugacy class of a proper free factor relative to the decomposition (the finite generation hypothesis on $H$ can be dropped for $G=F_N$, or more generally when $G$ is toral relatively hyperbolic).
In the first case, either $H$ virtually preserves a nonperipheral conjugacy class in $G$, or else $H$ contains an atoroidal automorphism. The key geometric tool to obtain these classification results is a description of the Gromov boundaries of relative versions of the free factor graph $\FF$ and the $\Z$-factor graph $\ZF$, as spaces of equivalence classes of arational trees (respectively relatively free arational trees). We also identify the loxodromic isometries of $\FF$ with the fully irreducible elements of $\Out(G,\calf)$, and loxodromic isometries of $\ZF$ with the fully irreducible atoroidal outer automorphisms.
\end{abstract}

\section*{Introduction}

Ivanov's celebrated classification theorem \cite{Iva92} for subgroups of the mapping class group of a compact, connected, hyperbolic surface $S$, associates to any subgroup $H$ of $\text{Mod}(S)$ a maximal decomposition of $S$ into proper  essential subsurfaces, which are 
invariant up to isotopy under a finite index subgroup $H^0$ of $H$. For each subsurface
 $\Sigma$ of the decomposition, either $H^0$ contains a mapping class whose restriction to $\Sigma$ is pseudo-Anosov, or else $H^0$ has trivial image in the mapping class group of $\Sigma$. The proof of this decomposition theorem has two steps:
\begin{enumerate}
\item One first shows that any subgroup of  $\text{Mod}(S)$ either contains a pseudo-Anosov mapping class of $S$, or else virtually fixes the isotopy class of a curve $c$ on $S$.
\item When there is an invariant curve $c$, one then argues by induction on the topological complexity of the surface, working in each of the connected components of $S\setminus c$. 
\end{enumerate} 

Finding an analogue of Ivanov's decomposition theorem for the group $\text{Out}(F_N)$ of outer automorphisms of a finitely generated free group has been a fruitful topic over the past decade. A first classification result for finitely generated subgroups of $\text{Out}(F_N)$ was established by Handel--Mosher in \cite{HM09}, and then extended to arbitrary subgroups of $\text{Out}(F_N)$ in \cite{Hor14-4}: every subgroup of $\text{Out}(F_N)$ either contains a fully irreducible automorphism, or else virtually preserves the conjugacy class of a proper free factor of $F_N$ (this is the analogous statement to Step~1 from the surface case). While Handel--Mosher's proof relies on train-track theory, the approach in \cite{Hor14-4} consists in using 
the action of $\text{Out}(F_N)$ on a hyperbolic graph, and a study of harmonic measures on the boundary of Culler--Vogtmann's outer space.
However, the inductive step from the surface case (Step~2 above) is more intricate for free groups: in order to get a full decomposition theorem, one indeed also needs to understand subgroups of $\Out(F_N)$ made of automorphisms that preserve a system of proper free factors of $F_N$.

A full decomposition theorem for finitely generated subgroups of $\Out(F_N)$ was obtained more recently by Handel--Mosher \cite{HM20}, with a proof again relying on 
train-track theory: this basically associates to every finitely generated subgroup $H$ of $\text{Out}(F_N)$ a maximal, virtually $H$-invariant filtration of $F_N$ by nested systems of free factors, such that for any extension $A\sqsubseteq B$ in this filtration (except possibly from a few sporadic cases), there exists an automorphism in $H$ that does not virtually preserve any intermediate free factor.

In the present paper, we generalize Handel--Mosher's full decomposition theorem to the context of automorphisms of free products; our new approach also enables us to remove the finite generation assumption in the full decomposition theorem for free groups. Our approach is similar to \cite{Hor14-4} and consists in classifying the subgroups of the automorphism group of a free product with respect to the dynamics of their action on a relative version of the free factor graph. A key ingredient in our proof is to understand the Gromov boundary of this graph. 

We finally mention that very recently, using the main results from the present paper, Clay and Uyanik have established \cite{CU} another alternative for subgroups of $\Out(F_N)$, namely: every subgroup of $\Out(F_N)$ either contains an atoroidal automorphism, or else virtually fixes a nontrivial conjugacy class of $F_N$.

\paragraph*{Subgroup classification for automorphisms of free products.}

Our setting is the following. Let $G$ be a countable group that splits as a free product $G:=G_1\ast\dots\ast G_k\ast F_N$, where all subgroups $G_i$ are nontrivial, and $F_N$ is a free group of rank $N$. This might be for example a Grushko decomposition of a finitely generated group (i.e.\ a maximal decomposition of $G$ as a free product), but it need not be: the case where $G$ is a free group and all subgroups $G_i$ are free factors is also of interest, and is precisely what one needs to carry out inductive arguments.
We denote by $\calf=\{[G_1],\dots,[G_k]\}$ the collection of all $G$-conjugacy classes of the subgroups $G_i$. Elements or subgroups of $G$ that are conjugate into one of the factors $G_i$ are called \emph{peripheral}. 
Throughout this introduction, we will assume that either $k+N\ge 3$, or else $G=F_2$; in other words, in addition to trivial cases, we exclude the so called \emph{sporadic decompositions}
of the form $G=G_1*G_2$ with $\calf=\{[G_1],[G_2]\}$ or $G=G_1*\bbZ$ with  $\calf=\{[G_1]\}$.
We denote by $\Out(G,\calf)$ the subgroup of $\Out(G)$ made of all outer automorphisms $\Phi$ such that for every $i\in\{1,\dots,k\}$, one has $\Phi([G_i])=[G_i]$. 

There is a notion of a free factor of $G$ with respect to the above free product decomposition, which is as follows. First, one defines a \emph{$(G,\calf)$-free splitting} 
as a nontrivial, minimal, simplicial tree, equipped with a simplicial $G$-action, in which each peripheral subgroup $G_i$ fixes a point $x_i$, and all edge stabilizers are trivial. A \emph{$(G,\mathcal{F})$-free factor} is then a subgroup of $G$ which arises as a point stabilizer in some $(G,\calf)$-free splitting. We say that a $(G,\calf)$-free factor is \emph{proper} if it is nontrivial, not conjugate into any subgroup $G_i$, and not equal to $G$ (this last condition being automatic if one defines, as we did, free splittings to be nontrivial).
An automorphism $\Phi\in\text{Out}(G,\mathcal{F})$ is \emph{fully irreducible} if no power $\Phi^k$ (with $k\neq 0$) preserves the conjugacy class of a proper $(G,\mathcal{F})$-free factor. Our first main result is the following (see Theorem~\ref{alternative-HM} for a more complete statement).

\begin{thmbis}\label{intro-alt-1}
Let $(G,\calf)$ be a nonsporadic free product, and let $H$ be a subgroup of $\text{Out}(G,\mathcal{F})$. Assume either that $H$ is finitely generated, or else that $G$ is a toral relatively hyperbolic group (e.g.\ a finitely generated free group).
\\ Then either 
\begin{enumerate}
\item[(1)] $H$ virtually preserves the conjugacy class of a proper $(G,\calf)$-free factor, or else 
\item[(2)] $H$ contains a fully irreducible outer automorphism; in this case, either 
\begin{enumerate}
\item[(2a)] $H$ contains a nonabelian free subgroup in which all nontrivial elements are fully irreducible, or else
\item[(2b)] $H$ is virtually a semidirect product  
$H^1\rtimes\langle\Phi\rangle$, where $\Phi$ is fully irreducible and $H^1$ contains no fully irreducible element.
\end{enumerate}
\end{enumerate}
\end{thmbis}

\begin{rk*}
  If $(G,\calf)=(F_N,\es)$ and $H$ satisfies Assertion (2b) then $H$ is virtually cyclic \cite[Corollary~1.3]{KL}, but this is not true
in general (see Remark \ref{rk_vcyclique}).
\end{rk*}

\begin{rk*} When $G$ is a toral relatively hyperbolic group, the factors
$G_1,\dots, G_k$ are automatically finitely generated, but in general,
we only assume that they are countable.
\end{rk*}

\indent We also prove the following variation over Theorem~\ref{intro-alt-1}, extending a theorem of Uyanik \cite{Uya14} (see Theorem~\ref{thm_hyperbolic2} for a more complete statement). An automorphism $\Phi\in\text{Out}(G,\mathcal{F})$ is \emph{atoroidal} if no power $\Phi^k$ (with $k\neq 0$) preserves a nonperipheral conjugacy class. 

\begin{thmbis}\label{intro-alt-2}
Let $(G,\calf)$ be a nonsporadic free product, and let $H\subseteq\text{Out}(G,\mathcal{F})$ be a subgroup. Assume either that $H$ is finitely generated, or else that $G$ is a toral relatively hyperbolic group (e.g.\ a finitely generated free group).
\\ Then either 
\begin{enumerate}
\item $H$ virtually preserves a nonperipheral conjugacy class of $G$, or the conjugacy class of a proper $(G,\calf)$-free factor, or else 
\item $H$ contains a fully irreducible atoroidal outer automorphism; in this case, either 
\begin{enumerate}
\item[(2a)] $H$ contains a nonabelian free subgroup in which all nontrivial elements are fully irreducible and atoroidal, or else
\item[(2b)] $H$ is virtually a semidirect product $H^1\rtimes\langle\Phi\rangle$, where $\Phi$ is fully irreducible and atoroidal,
and $H^1$ contains no fully irreducible element and no atoroidal element.
\end{enumerate}
\end{enumerate}
\end{thmbis}

If Assertion 1 holds but $H$ does not preserve a proper $(G,\calf)$-free factor, then
in fact $H$ fixes a \emph{quadratic} conjugacy class as in Definition~\ref{def-quadratic} (see also Lemma \ref{quadratic})
which roughly means that it comes from a boundary curve of a certain $2$-orbifold. 
See also Theorem~\ref{thm_Uyanik} where, in the case of the free group with $\calf=\es$, 
this is interpreted as saying that $H$ is contained in an extended mapping class group, a result due to Uyanik \cite{Uya14}.
\\
\\
\indent Our proof of Theorem~1 relies on understanding the dynamics of $\Out(G,\calf)$ on a hyperbolic graph, namely the \emph{free factor graph} $\FF$, and requires describing the points in the Gromov boundary of $\FF$ in terms of certain $G$-actions on $\mathbb{R}$-trees (points in the boundary correspond to equivalence classes of \emph{arational trees}, as explained later in this introduction). From this, the proof goes on using an argument from \cite{Hor14-5} 
(and related to \cite{KM96})
to prove Theorem~1 in the particular case where $G=F_N$ and $\calf=\emptyset$, and further used in \cite{Hor14-3} to prove a Tits alternative for automorphism groups of free products. We now briefly explain this argument.

The key point is to prove that if $H\subseteq\Out(G,\calf)$ is a subgroup, then either $H$ virtually fixes  the conjugacy class of a proper $(G,\calf)$-free factor, or else $H$ has unbounded orbits in $\FF$. In the second case, using Gromov's classification of isometric group actions on hyperbolic spaces, we deduce that either $H$ contains a purely loxodromic nonabelian free subgroup, or else that $H$ virtually fixes a point in $\partial_\infty\FF$ -- and we understand stabilizers of arational trees using work of Levitt and the first author \cite{GL}.

To prove the key point, let $\mu$ be a probability measure supported on $H$. Let $\nu$ be a $\mu$-harmonic  
probability measure on the closure of the corresponding projectified outer space (see below for definitions). If $\nu$ is concentrated on the space of arational trees, by projecting to the free factor graph, we see that a typical random walk on $H$ will diverge towards infinity in $\FF$; in particular $H$-orbits in $\FF$ are unbounded. Otherwise, as one can canonically associate to any nonarational tree a finite set of proper $(G,\calf)$-free factors, we get a $\mu$-stationary measure $\overline{\nu}$ on the countable set of all these free factors; the finite collection of all factors with maximal $\overline{\nu}$-measure is then $H$-invariant.
\\
\\
\indent Our proof of Theorem~2 is similar, using the dynamics of $\Out(G,\calf)$ on another hyperbolic graph, the \emph{$\calz$-factor graph} $\ZF$. The rest of this introduction is devoted to presenting the graphs $\FF$ and $\ZF$ and our description of their boundaries.

\paragraph*{Hyperbolic $\Out(G,\calf)$-graphs.}

In recent years, there has been a huge development of hyperbolic complexes associated with actions of outer automorphisms of free groups or free products, some of them having many quasi-isometric versions. We hope that the present paper will help clarifying the global picture, and the relationships between all these complexes. We now quickly describe the four graphs we need, and refer to Section~\ref{sec-graphs} for a detailed discussion.

The four graphs we use are the  \emph{free splitting graph} $\FS$,
the \emph{$\calz$-splitting graph} $\ZS$,
   the \emph{free factor graph} $\FF$ and the \emph{$\calz$-factor graph} $\ZF$. As we will see, they fit in the following diagram:
       $$\xymatrix{\FS \ar[r]\ar[d] &\ZS\ar[d]\ar[dl]_j \\
       \FF\ar[r] & \ZF}$$
where all the maps are coarsely surjective, coarsely alignment-preserving, all of them except $j$ being induced by inclusions.

We recall that the \emph{free splitting graph} $\FS$ is the graph whose vertices are the 
$(G,\calf)$-free splittings (up to equivariant homeomorphism), where two free splittings are joined by an edge if they have a common refinement. Hyperbolicity of the free splitting graph was proved by Handel--Mosher (\cite{HM12} for free groups, \cite{HM14} for free products). If one replaces free splittings by $\calz$-splittings of $(G,\calf)$ in the definition 
(i.e.\ simplicial, minimal nontrivial $(G,\calf)$-trees in which all edge stabilizers are either trivial or cyclic and nonperipheral), then one gets the \emph{$\calz$-splitting graph} $\ZS$, whose hyperbolicity was proved in \cite{Man12} for free groups and in \cite{Hor14-2} for free products\footnote{\textbf{Notational warning:} In \cite{Man12,Hor14-2}, the $\calz$-splitting graph is denoted as $FZ$; however, the notation $\ZS$ looked better to us: in this way, $\FS$ and $\ZS$ are the two splitting graphs, and $\FF$ and $\ZF$ will be the corresponding factor graphs, which naturally arise as their electrifications.}.

A convenient definition of the \emph{free factor graph} $\FF$ is the following `electrification' of the free splitting graph $\FS$:
the graph $\FF$ is the graph obtained from $\FS$ by adding, for each proper $(G,\calf)$-free factor $A\subseteq G$, 
an edge between any two free splittings in which $A$ is elliptic.
Equivalently, the vertices of $\FF$ are the $(G,\calf)$-free splittings, two splittings being joined by an edge if they are compatible or share a common nonperipheral elliptic element. 
Our definition of the free factor graph is quasi-isometric to the versions considered by Hatcher--Vogtmann \cite{HV} (or its natural adaptation in \cite{BF14})
and Handel--Mosher in \cite{HM14}, 
except in low-complexity cases where our definition needs no special treatment.
Hyperbolicity of the free factor graph was proved by Bestvina--Feighn \cite{BF14} for free groups, and by Handel--Mosher \cite{HM14} for free products (apart from one low-complexity case). 

Analogously, the \emph{$\calz$-factor graph} $\ZF$ is obtained by `electrifying' $\ZS$: 
$\ZF$ is the graph whose vertices are the (homeomorphism classes of) $\calz$-splittings of $(G,\calf)$, 
two splittings being joined by an edge if they are compatible or share a common nonperipheral elliptic element.

The inclusions between the four graphs yield the horizontal and vertical arrows in the diagram above. Using the map $\FF\ra \ZF$ and an argument by Kapovich-Rafi \cite{KR14}, we show that the $\calz$-factor graph is Gromov hyperbolic (see \cite{Man14} for free groups).
In the context of free groups, the $\calz$-factor graph turns out to be quasi-isometric to the intersection graph in the version of Mann \cite{Man14} or to Dowdall--Taylor's co-surface graph \cite{DT16}. 

The graph $\FF$ can also be viewed as an electrification of $\ZS$.
Recall that an element is \emph{simple} if it is contained in a proper $(G,\calf)$-free factor.
The graph obtained from $\ZS$ by adding an edge when two splittings share a common nonperipheral \emph{simple}  elliptic element is quasi-isometric to $\FF$ (see Lemma \ref{lem_FF2}). 
Using this description, the inclusion yields the map $j:\ZS\ra\FF$.

As an aside, we also show in Section~\ref{sec-zmax} that we get a quasi-isometric graph if we only allow for splittings over trivial or maximally cyclic subgroups in the definition (this is different from what happens with splitting graphs, where using $\Z$-splittings or $\Zmax$-splittings yield different graphs up to quasi-isometry \cite{Hor14-2}).

\paragraph*{Gromov boundaries.} 
All the maps $$\FS \ra \ZS  \overset{j}{\ra}   \FF\ra  \ZF$$
are coarsely surjective and coarsely alignment-preserving.
Using a theorem of Dowdall--Taylor \cite[Proposition~3.2]{DT16},
this gives the following topological inclusions between their Gromov boundaries: 
$$\partial_\infty \FS \hookleftarrow \partial_\infty\ZS  \hookleftarrow   \partial_\infty\FF\hookleftarrow  \partial_\infty\ZF.$$
When describing $\partial_\infty\FF$ and $\partial_\infty\ZF$, we can thus take advantage of the description of $\partial_\infty\ZS$ given in \cite{Hor14-2}, and we only need to determine which subsets stay at infinity of $\FF$ and $\ZF$.
The boundary of $\FS$ is more complicated: one can show that there is no subset of $\ol\calo$ that would make
a statement analogous to Theorem~\ref{intro-bdy-ff} below work for $\partial_\infty\FS$, and more specifically that there is no
equivariant continuous closed surjective map from a subset of $\ol\calo$ to $\partial_\infty\FS$.
A description of $\partial_\infty\FS$ for free groups has been announced by Bestvina--Feighn--Reynolds.

The Gromov boundary of the free factor graph of a free group was described by Bestvina--Reynolds \cite{BR13} and independently Hamenstädt \cite{Ham13} in terms of a notion of arational trees in the boundary of Culler--Vogtmann's outer space. We extend their description to the context of free products.

To explain this description, we first recall the notion of the \emph{(unprojectivized) outer space} $\calo$ of a free product, as introduced by Levitt and the first author in \cite{GL07}. The space $\calo$ is the space of \emph{Grushko trees} (of $G$ relative to $\calf$), i.e.\ simplicial metric relatively free $(G,\calf)$-trees with trivial edge stabilizers (here a \emph{$(G,\calf)$-tree} is a $G$-action on a tree where all subgroups in $\calf$ are elliptic, and it is \emph{relatively free} if these are the only point stabilizers). Its closure $\overline{\mathcal{O}}$ in the space of nontrivial, minimal $G$-actions on $\mathbb{R}$-trees was identified in \cite{Hor14-1} as the so-called \emph{very small} $(G,\calf)$-trees. 
A tree $T\in\overline{\calo}\setminus\calo$ is said to be \emph{arational} if the action of every proper $(G,\mathcal{F})$-free factor on its minimal subtree in $T$ is simplicial and relatively free. We denote by $\mathcal{AT}$ the subspace of $\overline{\calo}$ made of arational trees. Two arational trees $T,T'$ are equivalent, which we denote by $T\simAT T'$, if there exist  $G$-equivariant alignment-preserving bijections
between them. In the following statement, the map $\psi_{\FF}$ is the  $\Out(G,\calf)$-equivariant map from $\calo$ to $\FF$ obtained by forgetting the metric.

\begin{thmbis}\label{intro-bdy-ff}
There is a unique $\text{Out}(G,\mathcal{F})$-equivariant homeomorphism 
$$\partial\psi_{\FF}:\mathcal{AT}/\!\raisebox{-.2em}{{$\simAT$}}\to\partial_{\infty}\FF$$ 
such that for all $T\in\mathcal{AT}$ and all sequences $(T_n)_{n\in\mathbb{N}}\in\mathcal{O}^{\mathbb{N}}$ converging to $T$, the sequence $(\psi_{\FF}(T_n))_{n\in\mathbb{N}}$ converges to $\partial\psi_{\FF}(T)$. 
\\ Moreover, if $(T_n)_{n\in\mathbb{N}}\in\mathcal{O}^{\mathbb{N}}$ converges to $T\in\ol\calo$ which is not arational, then $\psi_{\FF}(T_n)$ has no accumulation point in $\partial_{\infty}\FF$.
\end{thmbis}

Bestvina--Reynolds's and Hamenstädt's proofs of this statement in the context of free groups (i.e.\ for $G=F_N$ and $\calf=\emptyset$) rely on a unique duality property of arational $F_N$-trees, namely, the equivalence class of an arational tree is determined by its dual current. This duality statement allows for a Kobayashi-type argument showing that arational trees are `at infinity' in $\FF$. In our previous paper \cite{GH15-1}, we established a unique duality statement for arational trees, although this was phrased in terms of laminations as currents turn out to be badly adapted to the context of free products (see the discussion in \cite[Section~3]{GH15-1}). Again, this duality statement is used to show that sequences in $\calo$ converging to an arational tree have unbounded image in $\FF$.\\

We also describe the Gromov boundary of the $\calz$-factor graph $\ZF$, generalizing a theorem of Dowdall--Taylor \cite{DT16} for free groups.  
We denote by $\FAT$ the set of relatively free arational $(G,\calf)$-trees  (these are precisely the arational trees that do not come from a lamination on a  $2$-orbifold).  
In the following statement, $\psi_{\ZF}$ is the  $\Out(G,\calf)$-equivariant map from $\calo$ of $\ZF$ obtained by forgetting the metric.

\begin{thmbis}\label{intro-bdy-I}
There is a unique $\text{Out}(G,\mathcal{F})$-equivariant homeomorphism $$\partial\psi_{\ZF}:\FAT/\!\raisebox{-.2em}{$\simAT$}\to\partial_{\infty}\ZF$$ such that for all $T\in\FAT$ and all sequences $(T_n)_{n\in\mathbb{N}}\in\mathcal{O}^{\mathbb{N}}$ converging to $T$, the sequence $(\psi_{\ZF}(T_n))_{n\in\mathbb{N}}$ converges to $\partial\psi_{\ZF}(T)$. 
\\ Moreover, if $(T_n)_{n\in\mathbb{N}}\in\mathcal{O}^{\mathbb{N}}$ converges to a tree $T\in\overline{\calo}\setminus\FAT$, then $\psi_{\ZF}(T_n)$ has no accumulation point in $\partial_{\infty}\ZF$.
\end{thmbis}

\paragraph*{Loxodromic isometries of $\FF$ and $\ZF$.} We finally mention that our arguments also enable us to determine precisely which automorphisms in $\Out(G,\calf)$ act loxodromically on either the free factor graph or the $\calz$-factor graph. The following theorem generalizes previous work of Bestvina--Feighn \cite{BF14}, Mann \cite{Man14} and Gupta \cite{Gup2}.

\begin{thmbis}
An automorphism $\Phi\in\text{Out}(G,\calf)$ acts loxodromically on $\FF$ if and only if $\Phi$ is fully irreducible. It acts loxodromically on $\ZF$ if and only if it is fully irreducible and atoroidal. 
\end{thmbis}

\paragraph*{Organization of the paper.} Section~\ref{sec-1} reviews general background about automorphisms of free products and coarsely alignment-preserving maps between hyperbolic spaces. In Section \ref{sec-graphs}, we introduce various quasi-isometric versions of the graphs $\FF$ and $\ZF$, and prove their hyperbolicity. Section~\ref{sec-bdy} is devoted to the description of the Gromov boundaries of $\FF$ and $\ZF$. In Section~\ref{sec-loxo}, we use these descriptions to identify the loxodromic isometries of $\FF$ and $\ZF$. In Section~\ref{sec-bdd}, we show that every subgroup of $\Out(G,\calf)$ with bounded orbits in $\FF$ virtually fixes the conjugacy class of a proper $(G,\calf)$-free factor, and we also establish a similar result for subgroups having bounded orbits in $\ZF$. In Section~\ref{sec-5}, we study stabilizers of arational trees.
The proof of our two classification theorems (Theorem~1 and Theorem~2) is completed in Section~\ref{sec-appli}. 

\paragraph*{Acknowledgments.} 
The first author acknowledges support from the Institut Universitaire de France and from the Centre Henri Lebesgue (Labex ANR-11-LABX-0020-01). 
The second author acknowledges support from the Agence Nationale de la Recherche under Grant ANR-16-CE40-0006, and is grateful to the Fields Institute for its hospitality during the `Thematic Program on Teichmüller Theory and its Connections to Geometry, Topology and Dynamics' in Fall 2018, where the last stage of this project was completed.

\setcounter{tocdepth}{1}
\tableofcontents

\section{Background on free products and relative automorphisms}\label{sec-1}

\subsection{General definitions}

Let $G_1,\dots,G_k$ be a finite collection of nontrivial countable groups, let $F_N$ be a free group of rank $N$, and let $$G:=G_1\ast\dots\ast G_k\ast F_N.$$ We let $\mathcal{F}:=\{[G_1],\dots,[G_k]\}$ be the finite collection of all $G$-conjugacy classes of the subgroups $G_i$; we call it a \emph{free factor system} of $G$. The \emph{complexity} of the free product $(G,\calf)$ is defined as $\xi(G,\calf):=(k+N,N)$; complexities are ordered lexicographically. The free product $(G,\mathcal{F})$ is \emph{sporadic} if $\xi(G,\calf)\le (2,1)$, and \emph{nonsporadic} otherwise. Sporadic cases correspond to the following: either 
\begin{itemize}
\item $G=\{1\}$ and $\calf=\emptyset$ (equivalently $\xi(G,\calf)=(0,0)$), or
\item $G=G_1$ and $\calf=\{[G_1]\}$ (equivalently $\xi(G,\calf)=(1,0)$), or
\item $G=\mathbb{Z}$ and $\calf=\emptyset$ (equivalently $\xi(G,\calf)=(1,1)$), or
\item $G=G_1\ast G_2$ and $\calf=\{[G_1],[G_2]\}$ (equivalently $\xi(G,\calf)=(2,0)$), or
\item $G=G_1\ast\mathbb{Z}$ and $\calf=\{[G_1]\}$ (equivalently $\xi(G,\calf)=(2,1)$).
\end{itemize}
Subgroups or elements of $G$ which are conjugate into one of the subgroups of $\mathcal{F}$ will be called \emph{peripheral}.
We denote by $\text{Out}(G,\mathcal{F})$ the subgroup of $\text{Out}(G)$ made of all outer automorphisms $\Phi$ such that for every $i\in\{1,\dots,k\}$, one has $\Phi([G_i])=[G_i]$. We denote by $\text{Out}(G,\mathcal{F}^{(t)})$ the subgroup of $\text{Out}(G)$ made of all outer automorphisms which have a representative in $\Aut(G)$ that acts as a conjugation by an element $g_i\in G$ on each peripheral subgroup $G_i$.

By a theorem of Kurosh \cite{Kur}, every subgroup $A\subseteq G$ decomposes as a free product $A=(\ast_j H_j)\ast F$, 
where each $H_j$ is peripheral 
(there might be infinitely many $H_j$ in general), and $F$ is a free group (maybe of infinite rank)
and where a subgroup of $A$ is peripheral if and only if it is $A$-conjugate into some $H_j$. 
We denote by $\calf_{|A}$ the collection of all $A$-conjugacy classes of the subgroups $H_j$ from the above decomposition of $A$.

\subsection{Actions on trees, outer space and its closure}

A \emph{$(G,\mathcal{F})$-tree} is an $\mathbb{R}$-tree $T$ equipped with an isometric action of $G$, in which all peripheral subgroups of $G$ are elliptic (i.e.\ each of them fixes a point in $T$). The $G$-action on $T$ is \emph{trivial} if $G$ fixes a point, \emph{minimal} if $T$ does not contain any proper nonempty $G$-invariant subtree. The $G$-action on $T$ is \emph{relatively free} if all elliptic subgroups in $T$ are peripheral. 

A \emph{Grushko $(G,\mathcal{F})$-tree} is a minimal, simplicial metric relatively free $(G,\mathcal{F})$-tree with trivial edge stabilizers. Two Grushko $(G,\mathcal{F})$-trees are \emph{equivalent} if there exists a $G$-equivariant isometry between them. The \emph{unprojectivized outer space} $\mathcal{O}$ is defined \cite{GL07} to be the space of all equivalence classes of Grushko $(G,\mathcal{F})$-trees. The \emph{projectivized outer space} $\mathbb{P}\mathcal{O}$ is defined as the space of homothety classes of trees in $\mathcal{O}$. The spaces $\calo$ and $\mathbb{P}\calo$ come equipped with right actions of $\text{Out}(G,\mathcal{F})$, given by precomposing the actions (these can be turned into left actions by letting $\Phi.T:=T.\Phi^{-1}$ for all $T\in \mathcal{O}$ and all $\Phi\in\text{Out}(G,\mathcal{F})$). The closure $\overline{\mathcal{O}}$ of outer space in the space of all $G$-actions on $\mathbb{R}$-trees, equipped with the Gromov--Hausdorff equivariant topology, was identified in \cite{Hor14-1} with the space of all \emph{very small} $(G,\calf)$-trees, i.e.\ trees whose arc stabilizers are either trivial, or cyclic, root-closed and nonperipheral, and whose tripod stabilizers are trivial. Its projectivization $\mathbb{P}\overline{\calo}$ is compact, see \cite[Theorem~4.2]{CM87} and \cite{Hor14-1}.

A \emph{generalized branch point} in a $(G,\calf)$-tree  is a point $x\in T$ which is either a branch point (i.e.\ $T\setminus\{x\}$ has at least $3$ connected components) or an inversion point (i.e.\ a point $x\in T$ such that $T\setminus \{x\}$ has exactly $2$ connected components, and these
components are exchanged by $G_x$).
Notice that for $T\in\ol\calo$, points with nontrivial stabilizer in $T$ are generalized branch points; inversion points
may only occur if one of the peripheral groups $G_i$ is isomorphic to $\mathbb{Z}/2\mathbb{Z}$.
By \cite{Lev94} (see also \cite[Theorem~4.16]{Hor14-1} for free products), every tree $T\in\overline{\calo}$ splits in a unique way as a graph of actions in the sense of \cite{Lev94}, in such a way that
\begin{itemize}
\item vertices of the decomposition correspond to orbits of connected components of the closure of the set of generalized branch points, 
\item edges of the decomposition correspond to orbits of maximal arcs whose interior contain no generalized branch point. 
\end{itemize}
In particular, vertex groups act with dense orbits on the corresponding subtree of $T$ (maybe a point). The Bass–-Serre tree of the underlying graph of groups is very small (maybe trivial); it is called the \emph{Levitt decomposition} of $T$.

Notice that whenever a group $G$ acts on an $\mathbb{R}$-tree and contains a hyperbolic element, 
there is a unique subtree of $T$ on which the $G$-action is minimal. In particular, if $H$ is a subgroup of $G$ 
containing a hyperbolic element or fixing a unique point, then the $H$-action on $T$ admits a minimal subtree, which we call the \emph{$H$-minimal subtree} of $T$ (which we denote by $T_H$). The action of $H$ on $T$ is \emph{simplicial} if the $H$-minimal subtree is isometric to a simplicial tree.

\subsection{Maps between trees}

From now on, all maps between $G$-trees will be $G$-equivariant.

\paragraph*{Alignment-preserving maps.} Given two $(G,\mathcal{F})$-trees $T$ and $T'$, a map $f:T\to T'$ is \emph{alignment-preserving} if the $f$-image of every segment in $T$ is a segment in $T'$ (segments are allowed to be reduced to a point). If there exists a $G$-equivariant alignment-preserving map from $T$ to $T'$, we say that $T$ is a \emph{refinement} of $T'$. If $T$ and $T'$ are simplicial, we also say that $T'$ is a \emph{collapse} of $T$ and that $f$ is a \emph{collapse map} (topologically, $T'$ may be obtained from $T$ by collapsing a $G$-invariant collection of edges).
 Two trees are \emph{compatible} if they have a common refinement. 

We will use the following fact several times. 
\begin{lemma}[{\cite[Lemma 6.11]{Hor14-1}}]\label{lem_tirer}
  Let $T$ be a minimal, simplicial $(G,\calf)$-tree, whose edge stabilizers are all cyclic (they may be finite or peripheral). 
\\ Then $T$ is compatible with a $(G,\calf)$-free splitting.
\end{lemma}

In particular, every edge stabilizer of $T$ is contained in a proper $(G,\calf)$-free factor because
it is elliptic in any splitting compatible with $T$.

\paragraph*{Morphisms and folding paths.} Let $T$ and $T'$ be two $(G,\calf)$-trees. A \emph{morphism} $f:T\to T'$ is a map such that every segment of $T$ can be subdivided into finitely many subsegments, in such a way that $f$ is an isometry when restricted to any of these subsegments. A morphism $f:T\to T'$ is \emph{optimal} if every point $x\in T$ is contained in the interior of a segment $I$ such that $f_{|I}$ is an isometry. A \emph{folding path guided by $f$} is a continuous family $(T_t)_{t\in\mathbb{R}_+}$ of trees, together with a collection of morphisms $f_{t_1,t_2}:T_{t_1}\to T_{t_2}$ for all $0\le t_1<t_2$, such that 

\begin{itemize}
\item there exists $L\in\mathbb{R}$ such that for all $t\ge L$, we have $T_t=T'$, and
\item we have $f_{0,L}=f$, and 
\item for all $0\le t_1<t_2<t_3$, we have $f_{t_1,t_3}=f_{t_2,t_3}\circ f_{t_1,t_2}$.
\end{itemize}

\noindent Given two $(G,\mathcal{F})$-trees $T$ and $T'$, a \emph{folding path} from $T$ to $T'$ is a folding path guided by some morphism $f:T\to T'$. It is \emph{optimal} if $f$ is optimal (this implies that $f_{t_1,t_2}$ is optimal for all $t_1<t_2$).

\subsection{Coarsely alignment-preserving maps between hyperbolic graphs}\label{sec-dt}

Let $X$ and $Y$ be two geodesic metric spaces. Let $K\ge 0$. A triple of points $(a,b,c)\in X^3$ is \emph{$K$-aligned} if 
$ d_X(a,b)+d_X(b,c)\leq d_X(a,c)+K$ (if $(a,b,c)$ are $0$-aligned, we just say that they are \emph{aligned}). A map $\phi:X\to Y$ is \emph{coarsely alignment-preserving} if there exists $K\ge 0$ such that $\phi$ maps triples of aligned points to triples of $K$-aligned points. It is \emph{coarsely surjective} if there exists $K\ge 0$ such that for all $y\in Y$, we have $d_Y(y,\phi(X))\le K$. 
The following result of Kapovich--Rafi gives a criterion for checking the existence of a coarsely alignment-preserving map between two connected graphs $X,Y$, where $X$ is Gromov hyperbolic, and enables at the same time to deduce hyperbolicity of $Y$ from hyperbolicity of $X$.

\begin{prop}(Kapovich--Rafi \cite[Proposition 2.5]{KR14})\label{KR}
Let $(X,d_X)$ and $(Y,d_Y)$ be connected graphs, so that $X$ is Gromov hyperbolic. Assume that there exists a Lipschitz map $\phi:X\to Y$ sending vertices to vertices and edges to edge paths, and surjective on vertices, and that there exists $K>0$ such that for all $x,x'\in X$, if $d_Y(\phi(x),\phi(x'))\le 1$, then the $\phi$-image of any geodesic segment joining $x$ to $x'$ has $d_Y$-diameter at most $K$.\\
Then $Y$ is Gromov hyperbolic, and $\phi$ is coarsely alignment-preserving. 
\end{prop}

Following \cite{DT16}, if $X$ and $Y$ are Gromov hyperbolic geodesic metric spaces, and $\phi:X\to Y$ is coarsely alignment-preserving, we define $\partial_YX$ as the subspace of $\partial_\infty X$ made of all equivalence classes of  quasigeodesic rays whose $\phi$-image is unbounded in $Y$. The following result of Dowdall--Taylor identifies the Gromov boundary of $Y$ with this subspace of $\partial_\infty X$.

\begin{theo}(Dowdall--Taylor \cite[Theorem 3.2]{DT16})\label{dt}
Let $X,Y$ be two Gromov hyperbolic geodesic metric spaces, and let $\phi:X\to Y$ be a coarsely surjective, coarsely alignment-preserving map.
\\ Then there exists a homeomorphism $$\partial\phi:\partial_YX\to\partial_\infty Y$$ such that for all $\xi\in\partial_YX$ and all sequences $(x_n)\in X^\mathbb{N}$ converging to $\xi$, the sequence $(\phi(x_n))_{n\in\mathbb{N}}$ converges to $\partial\phi(\xi)$.
\\ Furthermore, if $(x_n)_{n\in\mathbb{N}}\in X^\mathbb{N}$ converges to some $\xi\in\partial_\infty X\setminus\partial_YX$, then $(\phi(x_n))_{n\in\mathbb{N}}$ has no accumulation point in $\partial_\infty Y$.
\end{theo}

\begin{rk}\label{rk-dt}
The second assertion is not explicitly stated in \cite{DT16}. It follows from the following useful observation.
\end{rk}

\begin{lemma}\label{lem_eqv_bord}
In the situation of Theorem \ref{dt},
  given $\xi\in\partial_\infty X$, the following are equivalent:
  \begin{enumerate}
  \item $\xi\in \partial_Y X$, i.e.\ any quasi-geodesic ray in $X$ representing $\xi$ has unbounded image in $Y$ under $\phi$,
  \item for any sequence $(x_n)_{n\in\mathbb{N}}\in X^\mathbb{N}$ converging to $\xi$, $\phi(x_n)$ is unbounded,
  \item for any sequence $(x_n)_{n\in\mathbb{N}}\in X^\mathbb{N}$ converging to $\xi$, $\phi(x_n)$ converges to a point in $\partial_\infty Y$ (namely $\partial \phi(\xi)$),
  \item there exists a sequence $(x_n)_{n\in\mathbb{N}}\in X^\mathbb{N}$ converging to $\xi$ such that $\phi(x_n)$ has an accumulation point in $\partial_\infty Y$.
  \end{enumerate}
\end{lemma}

\begin{proof}
Clearly, $3\imp 2\imp 1$ and the first assertion of Theorem \ref{dt} says $1\imp 3$.
Since $3\imp 4$, it suffices to prove $4\imp 3$. 
Let $\rho$ be a quasigeodesic ray in $X$ converging to $\xi$. 
If $\xi\in \partial_Y X$, then the first assertion of Theorem \ref{dt} says that $3$ holds.
So we assume that $\phi(\rho)$ is bounded in $Y$, and prove that for any sequence $x_n$ converging to $\xi$, $\phi(x_n)$ has no accumulation point.
Then up to extracting a subsequence, 
one can assume that for all $n\neq m$ a geodesic segment $[x_n,x_m]_X$ contains a point at bounded distance from $\rho$. 
It follows that the set of Gromov products $(\phi(x_n)|\phi(x_m))$ with $n\neq m$ is bounded. This implies that $(\phi(x_n))_{n\in\mathbb{N}}$ has no accumulation point in $\partial_\infty Y$. 
\end{proof}

\begin{rk}
When $X$ and $Y$ are endowed with an isometric action of a group $\Gamma$, and the map $\phi$ is coarsely $\Gamma$-equivariant, then the homeomorphism $\partial\phi$ is $\Gamma$-equivariant.
\end{rk}

\subsection{JSJ splittings}

We will use some arguments relying on the theory of JSJ decompositions of groups in Sections~\ref{sec-z} and~\ref{sec-zmax}. We now review some definitions from this theory (see \cite{GL16} for more details). Let $G$ be a group, and let $\calh$ be a collection of subgroups of $G$. 
A \emph{cyclic splitting} $S$ of $(G,\calh)$ is a minimal $(G,\calh)$-tree whose edge stabilizers are cyclic (we will allow finite and trivial cyclic groups).
A cyclic splitting of $(G,\calh)$ is \emph{universally elliptic} if for every cyclic splitting $S'$ of $(G,\calh)$, 
every edge stabilizer in $S$ is elliptic in $S'$. A \emph{cyclic JSJ splitting} of $(G,\calh)$ is a cyclic splitting of $(G,\calh)$ which is universally elliptic, and maximal for domination with respect to this condition (we recall that a splitting $S$ \emph{dominates} a splitting $S'$ if there exists a $G$-equivariant map $S\to S'$, or equivalently every point stabilizer in $S$ is elliptic in $S'$). 

Recall that $(G,\{[G_1],\dots,[G_k]\})$ is relatively finitely presented if there exists $r<\infty$ and $\calr$ a finite subset of $G_1*\dots *G_k *F_r$ 
such that $G\simeq (G_1*\dots *G_k *F_r)/\ngrp{\calr}$.
In particular, our free product $(G,\calf)$ is always relatively finitely presented.
If $(G,\calh)$ is relatively finitely presented
then cyclic JSJ decompositions do exist.

\section{Hyperbolic $\text{Out}(G,\calf)$-graphs}\label{sec-graphs}

In this section, we present three $\text{Out}(G,\calf)$-graphs, namely the $\mathcal{Z}$-splitting graph, the free factor graph and the $\calz$-factor graph. All graphs are equipped with the simplicial metric. We give several possible models for these graphs, and establish that the different models are all quasi-isometric. We provide proofs of the hyperbolicity of the free factor and the $\calz$-factor graph (Handel--Mosher have already given a proof of the hyperbolicity of the free factor graph in \cite{HM14}, however we give a proof that includes one more low-complexity case).

\subsection{The $\mathcal{Z}$-splitting graph: review}

A \emph{$\mathcal{Z}$-splitting} of $(G,\mathcal{F})$ is a nontrivial, minimal, simplicial $(G,\mathcal{F})$-tree, all of whose edge stabilizers are either trivial, or cyclic and nonperipheral. The \emph{$\mathcal{Z}$-splitting graph} $\ZS$ is the graph whose vertices are the equivariant homeomorphism classes of  $\mathcal{Z}$-splittings of $(G,\mathcal{F})$, two distinct vertices being joined by an edge if the corresponding splittings are compatible. The graph $\ZS$ admits a natural right action of $\text{Out}(G,\mathcal{F})$, by precomposition of the actions. The graph $\ZS$, equipped with the simplicial metric, is Gromov hyperbolic \cite{Man12,Hor14-2}.

\subsection{The free factor graph}\label{sec-ff}

\subsubsection{Quasi-isometric models}\label{sec-models-ff}
A \emph{free splitting} of $(G,\calf)$ is a nontrivial, minimal, simplicial $(G,\calf)$-tree, in which all edge stabilizers are trivial. We will take the following as our main definition of the free factor graph (this is different from the traditional one, which will be given in Definition~\ref{ff3} below and justifies the name \emph{free factor graph}). In the sequel, we will sometimes write $\FF$ to denote the free factor graph, identified with any quasi-isometric model. 
 
\begin{de}[\emph{\textbf{Free factor graph, version 1}}] \label{dfn_ff1}
The \emph{free factor graph} $\FF_1$ is the simplicial graph whose vertices are the free splittings of $(G,\calf)$, in which two vertices are joined by an edge if the corresponding splittings are compatible or have a common nonperipheral elliptic element.
\end{de}

Equivalently, two free splittings are joined by an edge in $\FF_1$ if and only if  they are compatible or there is a proper free factor which fixes a point in both of them.

\begin{rk}
In the case where $(G,\calf)$ is sporadic, one checks that $\FF_1$ is either empty or bounded. More precisely,
\begin{itemize}
\item if $\xi(G,\calf)\le (1,0)$, i.e.\ either $G=\{1\}$ or $G=G_1$ with $\calf=\{[G_1]\}$, then $\FF_1$ is empty;
\item if $\xi(G,\calf)=(1,1)$, i.e.\ $G=\mathbb{Z}$ and $\calf=\emptyset$, then $\FF_1$ is a point (corresponding to the 
action of $\bbZ$ on a line by translations);
\item if $\xi(G,\calf)=(2,0)$, i.e.\ $G=G_1\ast G_2$ and $\calf=\{[G_1],[G_2]\}$, then $\FF_1$ is a point (corresponding to this splitting of $G$);
\item if $\xi(G,\calf)=(2,1)$, i.e.\ $G=G_1\ast$ and $\calf=\{[G_1]\}$, then the graph $\FF_1$ is a star of diameter $2$ (the central vertex corresponds to the HNN extension $G=G_1\ast$, and all other vertices correspond to splittings of the form $G=G_1\ast\langle a\rangle$ with $a\in G$).
\end{itemize}
\end{rk}

In order to view $\FF$ as an electrification of the $\calz$-splitting graph $\ZS$, it will be convenient to use the following version of the free factor graph (we prove below that the two versions are quasi-isometric to each other). We recall that an element $g\in G$ is \emph{simple} if it is contained in some proper $(G,\calf)$-free factor.

\begin{de}[\textbf{\emph{Free factor graph, version 2}}] \label{dfn_ff2}
We define $\FF_2$ as the simplicial graph whose vertices are the  $\calz$-splittings of $(G,\calf)$, in which two splittings are joined by an edge if they are compatible or have a common nonperipheral \textbf{simple} elliptic element.
\end{de}

Notice that every element of $G$ which is elliptic in a free splitting of $(G,\calf)$ is simple. This provides a natural inclusion map $i:\FF_1\to \FF_2$.

\begin{lemma}\label{lem_FF2}
The inclusion map $i:\FF_1\to \FF_2$ is a quasi-isometry. 
\end{lemma}

\begin{proof}
The map $i$ is clearly Lipschitz; we will define a quasi-inverse of $i$. Given a $\calz$-splitting $T$ of $(G,\calf)$, we let 
$\Tilde \theta(T)$ be the set of all free splittings of $(G,\calf)$ that are compatible with $T$. By Lemma \ref{lem_tirer}, this set is nonempty. 
We claim that $\Tilde \theta(T)$ has bounded diameter in $\FF_1$. Indeed, if $T$ has a nontrivial edge stabilizer $E$, then $E$ is simple by Lemma~\ref{lem_tirer}, and it is elliptic in all trees in $\Tilde \theta(T)$, so $\Tilde \theta(T)$ has diameter at most 1.
If $T$ is a free splitting, then all trees in $\Tilde \theta(T)$ are at distance at most 1 from $T$ in $\FF_1$, so $\Tilde \theta(T)$ has diameter at most 2. We now define $\theta:\FF_2\ra \FF_1$ by sending a tree $T$ to an element in $\Tilde \theta(T)$.

The above argument shows that $\theta\circ i$ is at distance at most $1$ from the identity.
Similarly, if $T$ is a $\calz$-splitting, then $i\circ \theta(T)$ is a free splitting compatible with $T$, so $i\circ \theta$ is at distance at most $1$ from the identity. Therefore $\theta$ is a quasi-inverse of $i$.

To see that $\theta$ is Lipschitz, let $T_1$ and $T_2$ be two $\calz$-splittings of $(G,\calf)$ such that $d_{\FF_2}(T_1,T_2)=1$: we want to show that $d_{\FF_1}(\theta(T_1),\theta(T_2))$ is bounded. First assume that $T_1$ and $T_2$ are compatible, i.e.\ they have a common refinement $T$. Then by Lemma~\ref{lem_tirer}, there exists a free splitting $S$ compatible with $T$. Then $S$ is compatible with $T_1$ and $T_2$, showing that $\Tilde \theta(T_1)\cap\Tilde \theta(T_2)\neq \es$, so $\theta(T_1)$ and $\theta(T_2)$ are at bounded distance. The case where there exists a simple nonperipheral element $a\in G$ that is elliptic in both $T_1$ and $T_2$ is a consequence of the following lemma (extending Lemma~\ref{lem_tirer}), applied to both splittings $T_1$ and $T_2$.
\end{proof}

\begin{lemma}\label{lem_depliage}
Let $T$ be a minimal, simplicial $(G,\calf)$-tree
with cyclic or peripheral edge stabilizers, and let $a\in G$ be a nonperipheral simple element that is elliptic in $T$. 
\\ Then there exists a free splitting of $(G,\calf)$ that is compatible with $T$, in which $a$ is elliptic. 
\end{lemma}

\begin{proof} 
Since $a$ is simple, there exists a free splitting $S_0$ of $(G,\calf)$ with only one orbit of edges in which $a$ is elliptic. Let $\hat S$ be a blowup of $S_0$ obtained by looking at the action of the vertex stabilizers of $S_0$ on $T$: there is an equivariant map $f:\hat S\ra T$, isometric on the connected components of the complement of the edges coming from $S_0$. One can write $f$ as a composition of a collapse map and folds and note that all intermediate trees are minimal.

Let $E\subseteq \hat S$ be the set of all edges of $\hat{S}$ with trivial stabilizer that are collapsed by $f$, and let $S'$ be the tree obtained from $\hat{S}$ by collapsing all edges in $E$. Assume first that $S'$ contains no edge with trivial stabilizer. Then $S'$ is mapped
isometrically to $T$ by $f$: indeed, $f$ is an isometry when restricted to one component of the complement of the edges with trivial stabilizers in $\hat S$, and in addition two edges in distinct components have stabilizers which are not contained in a common elementary  (i.e.\ cyclic or peripheral) subgroup, so $f$ cannot fold two edges in distinct components without creating an edge with nonelementary stabilizer. Therefore $\hat S$ is compatible with $T$ in this case.

We now assume that $S'$ contains an edge with trivial stabilizer. Up to replacing $\hat S$ by $S'$ and subdividing $S'$ and $T$, we may assume that 
$f$ maps edges to edges and $T$ has no inversion (i.e.\ no element $g\in G$ flips the two endpoints of an edge of $T$).
We will write $f$ as a composition of folds and find a free splitting compatible with $T$ in which $a$ is elliptic.
Given a simplicial $(G,\calf)$-tree $S_1$ and two edges $e,e'$ incident on a common vertex $v$,
the \emph{fold} defined by $(e,e')$ is the quotient  $S_1\ra S_2$  obtained
from $S_1$ by equivariantly identifying $e$ with $e'$.
If $e$ is incident on $v$ and $H\subseteq G_v$, the \emph{fold} defined by $(e,H)$
is the quotient $S_1\ra S_2$  obtained
from $S_1$ by identifying all edges in $H.e$ and making this equivalence relation equivariant.
If $(e,e')$ are in the same $G_v$-orbit, say $e'=g.e$ for some $g\in G_v$, then folding $(e,e')$ is the same
as folding $(e,\grp{g})$. Moreover, if $e$ and $e'$ are in the same orbit, but not in the same $G_v$-orbit,  then $f$ cannot factor through the fold $(e,e')$ as otherwise $T$ would have an inversion.
Therefore, we can write $f$ as a concatenation of folds defined by $(e,H)$, and folds defined by $(e,e')$ where $e$ and $e'$ are in distinct orbits. 

We note that if $T'$ is obtained from $T$ by a fold, then $T$ and $T'$ are compatible,
a common refinement being obtained by folding the involved edges on half of their length.

We say that a fold is
\begin{itemize}
\item \emph{of type 1} if it is defined by a pair of edges $(e,e')$ in distinct orbits, and both have nontrivial stabilizer,
\item \emph{of type 2} if it is defined by a pair of edges $(e,e')$ in distinct orbits, and at least one of them has trivial stabilizer,
\item \emph{of type 3} if it is a fold defined by $(e,H)$. 
\end{itemize}
We construct inductively a folding path $S'=S_1\ra S_2\ra \dots \ra S_k$ through which $f$ factors as follows,
and satisfying the following maximality condition:
if some edge $e$ of $S_i$ has nontrivial stabilizer, then its stabilizer coincides with the stabilizer of its image in $T$.
We note that this condition is automatically satisfied by $S'$.
Now starting from $S_i$, if the map $S_i\ra T$ factors through a fold of type $j\in \{1,2,3\}$, but not of type $j'<j$, we perform such a fold to define $S_{i+1}$.
Additionally, we claim that if no fold of type 1 or 2 is possible, then one can perform
a fold of type $3$ defined by $(e,H)$ where $H\subseteq G_v$ is the full stabilizer of the image of $e$ in $T$.
This will guarantee that $S_{i+1}$ still satisfies the maximality condition.

To prove the claim, consider a pair of  edges $e,ge$ incident on $v$ and identified in $T$. Let $H$ be the stabilizer of the image of $e$ in $T$ (it contains $g$).
We know that $H$ fixes a vertex $u\in S_i$, either because $H$ is peripheral, or because $H$ is cyclic and thus contains $\grp{g}$ with finite index.
All edges in $[u,v]$ have nontrivial stabilizer (they are fixed by $g$). If $[u,v]$ is not mapped injectively in $T$, then one can perform
a fold, among two adjacent edges $(e'_1,e'_2)$ in $[u,v]$. Since there is no possible fold of type 1 or 2, $e'_2=he'_1$ for some $h\in H$, which is impossible
since $h$ fixes $u$. So $[u,v]$ is mapped injectively into $T$, and it follows that $H$ fixes the image of $[u,v]$ in $T$.
Since all edges of $[u,v]$ have nontrivial stabilizer, the maximality property shows that all edges of $[u,v]$ are fixed by $H$.
Thus, one can perform the fold $(e,H)$, and the claim is proved.

Now for some $k$, one must have that the map $S_k\ra T$ is an isomorphism since folds of type 1 or 2
decrease the number of orbits of edges, and folds of type 3 decrease the number of orbits of edges with trivial stabilizer.

Let $S_j$ be the last tree along the folding sequence that contains an edge with trivial stabilizer. 

We claim that $S_{j+1}=T$; this will conclude the proof as $T$ will then be compatible with $S_j$, and therefore with the free splitting determined by any edge of $S_j$ with trivial stabilizer (and $a$ is elliptic in this free splitting). 
We now prove the claim. The fold $f_j:S_j\to S_{j+1}$ is either of type 2 defined by a pair $(e,e')$ of edges, exactly one (say $e$) having nontrivial stabilizer, or it is of type 3, defined by $(e,H)$, and $e$ has trivial stabilizer. 
Assume towards a contradiction that $S_{j+1}\neq T$. 

Since all edges of $S_{j+1}$ have nontrivial stabilizer, our maximality condition implies that the fold $S_{j+1}\ra S_{j+2}$
is defined by a pair of edges $e_1,e_2\subset S_{j+1}$ in distinct orbits and having the same nontrivial stabilizer.
The preimage in $S_j$ of any edge of $S_{j+1}$ consists either in a single edge,
or in a set of edges having a common vertex. Assume first that for $i\in\{1,2\}$ the preimage of $e_i$ is a single edge $\tilde e_i\subset S_j$. 
In particular, $e_i$ and $\tilde{e_i}$ have the same stabilizer.
If $\tilde e_1$ and $\tilde e_2$ were adjacent, then one could have folded them together, contradicting that we cannot perform a fold of type 1 in $S_j$.
It follows that the path joining $\tilde e_1$ to $\tilde e_2$ is non-degenerate. This path has to contain two adjacent edges that are identified in $S_{j+1}$,
so at least one of them has trivial stabilizer.
This contradicts that $\tilde e_1$ and $\tilde e_2$ have the same non-trivial stabilizer.

Thus, we may assume that the preimage  of $e_1$ does not consist of a single edge.
We denote by $E_1$ the collection of edges that are mapped to $e_1$, and by $w$ their common vertex.
We note that $G_{e_1}=G_{e_2}$ fixes $w$.
Since $e_1$ and $e_2$ are not in the same orbit, the preimage of $e_2$ is a single edge $\tilde e_2$.
Consider $\tilde e_1$ in $E_1$, chosen so that $\tilde e_1$ has non-trivial stabilizer if $f_j$ is a fold of type $2$.
We claim that $\tilde e_1$ and $\tilde e_2$ are adjacent.
Otherwise, the path joining them contains a pair of folded edges, and hence an edge with trivial stabilizer; this contradicts the fact that $\tilde e_2$ and $w$ are both fixed by $G_{e_2}$. 

Now, if $f_j$ is of type 2, then $\tilde{e}_1$ and $\tilde{e}_2$ are two adjacent edges with nontrivial stabilizer that are identified in $T$, contradicting the fact that we could not perform a fold of type 1. If $f_j$ is of type 3, then $\tilde{e}_1$ and $\tilde{e}_2$ are identified in $T$ and one of them has nontrivial stabilizer, contradicting the  fact that we could not perform a fold of type 2.
\end{proof}

\paragraph{Hatcher--Vogtmann's and Handel--Mosher's models.}

In this paragraph, we relate our definition to Hatcher--Vogtmann's \cite{HV}  and Handel--Mosher's  \cite{HM14}. This explains why it is legitimate
to call it the free factor graph. This will not be used in the paper.

The natural adaptation of Hatcher--Vogtmann's definition from \cite{HV} is the following (notice that they were defining the complex of free factors as an $\Aut(F_N)$-complex, but there is a natural analogue of their definition for $\Out(F_N)$, as considered in \cite{BF14}).

\begin{de}[\emph{\textbf{Free factor graph, version 3}}]\label{ff3}
We define $\FF_3$ as the simplicial graph whose vertices are the conjugacy classes of nonperipheral proper $(G,\mathcal{F})$-free factors, in which two vertices $[A]$ and $[B]$ are joined by an edge whenever there are representatives in their conjugacy classes such that either $A\varsubsetneq B$ or $B\varsubsetneq A$. 
\end{de}

Handel--Mosher's definition from  \cite{HM14} is in terms of \emph{free factor systems}. We find convenient to describe it in terms of deformation spaces
of free splittings, but this is strictly equivalent. Recall that $T_1$ \emph{dominates} $T_2$ if there is an equivariant map $T_1\ra T_2$, or equivalently, if vertex stabilizers of $T_1$ are elliptic in $T_2$. One says that $T_1$ and $T_2$ \emph{are in the same deformation space}
if they dominate each other. In particular, two free splittings are in the same deformation space if the two free factor systems defined by
their collections of vertex stabilizers coincide.

\begin{de}[\emph{\textbf{Free factor graph, version 4}}]\label{ff4}
We define $\FF_4$ as the simplicial graph whose vertices are the deformation spaces of $(G,\calf)$-free splittings which are nontrivial and not relatively free, with an edge between two deformation spaces if one dominates the other.
\end{de}

\begin{prop}
The graph $\FF_3$ is quasi-isometric to $\FF_1$ for $\xi(G,\calf)\ge (3,2)$.
\\ The graph $\FF_4$ is quasi-isometric to $\FF_1$ in all nonsporadic cases except for $(G,\calf)=G_1\ast G_2\ast G_3$.
\end{prop}

We now sketch the proof of the second assertion of the proposition. The first assertion is similar and left to the reader.

Given a free splitting $T$ of $(G,\calf)$, we denote by 
$\Tilde\theta(T)\subseteq \FF_4$ the collection of all nontrivial, non relatively free deformation spaces obtained
by collapsing a  (possibly empty) $F_N$-invariant subset of edges of $T$.
Since $(G,\calf)$ is non-sporadic, $\Tilde\theta(T)$ is non-empty 
and we define $\theta(T)$ by choosing some element in $\Tilde \theta(T)$.

\begin{lemma}\label{lem_ff}
Assume that  $(G,\calf)$ is nonsporadic and $\xi(G,\calf)\neq (3,0)$.\\ 
Then there exists $C>0$ such that for every free splitting $T$ of $(G,\calf)$, the set $\Tilde\theta(T)$ 
has diameter at most $C$ in $\FF_4$. 
In addition $\theta$ defines a quasi-isometry between $\FF_1$ and $\FF_4$.
\end{lemma}

\begin{figure}[htb]
\centering
\includegraphics[width=\linewidth]{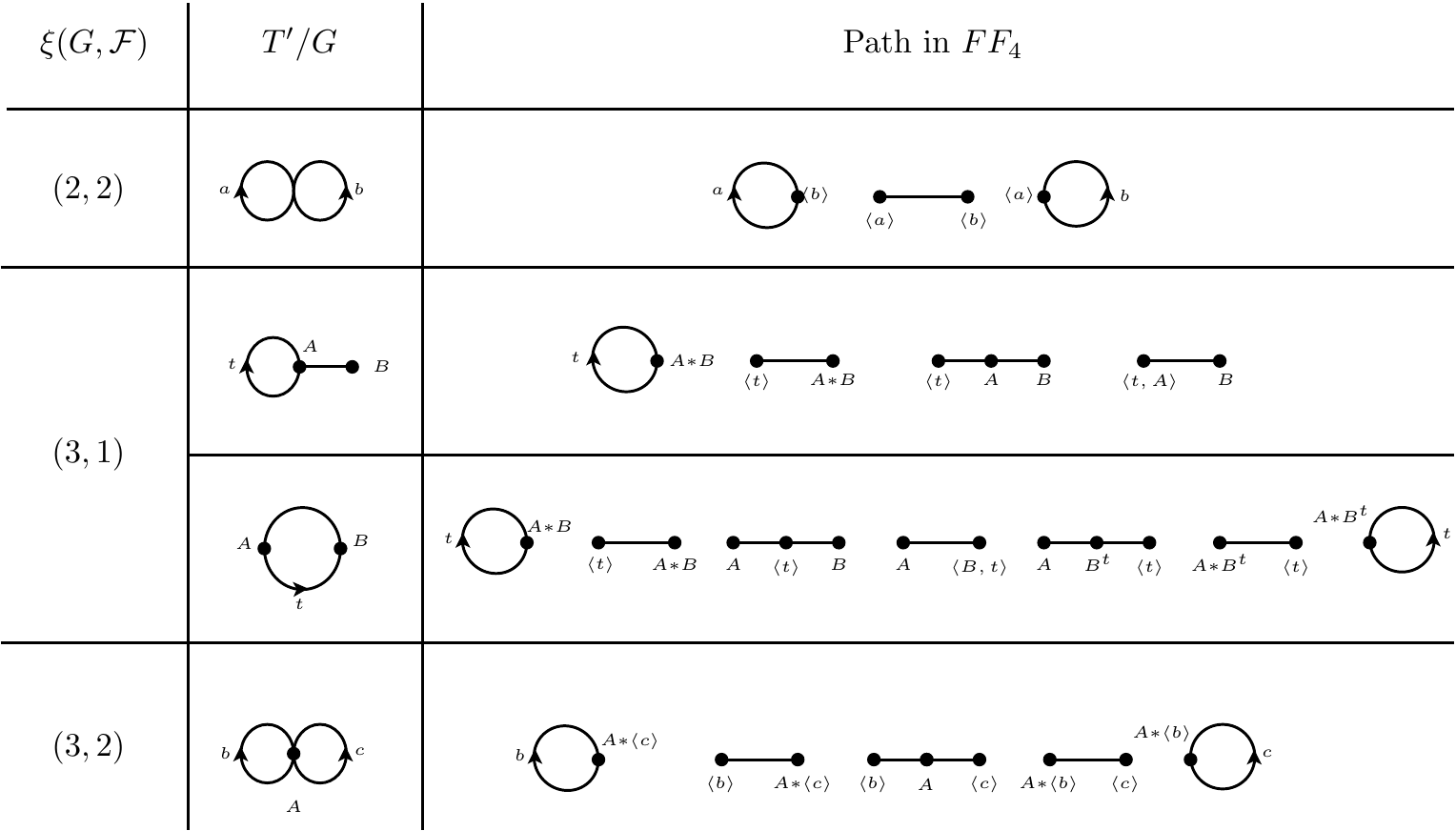}
\caption{Various cases in the proof of Lemma~\ref{lem_ff}.}
\label{fig:hm}
\end{figure}

\begin{proof}
Since $\Tilde\theta(\overline{T})\subseteq\Tilde\theta(T)$ whenever $\overline{T}$ is a collapse of $T$, we can assume without loss of generality that $T$ is a Grushko $(G,\calf)$-tree. 
Let $S_1,S_2$ be two trees obtained by collapsing some sets of edges $\Gamma_1,\Gamma_2\subseteq T/G$. 
Let $e_1$ (resp.\ $e_2$) be an edge in the complement of $\Gamma_1$ (resp.\ $\Gamma_2$) in $T/G$. Let $T'$ be the tree obtained from $T$ by collapsing all edges outside of the orbits of $e_1$ and $e_2$ to points. If $T'$ is not a Grushko tree, then it defines a vertex in $\FF_4$
at distance at most 2 from $S_1$ and $S_2$, and we are done.
If $T'$ is a Grushko tree, then there exists a Grushko $(G,\calf)$-tree with two orbits of edges. The only possibilities for $\xi(G,\calf)$ are thus $(2,2)$, $(3,1)$ and $(3,2)$ (as $(3,0)$ has been excluded). In each of these cases, we have depicted on Figure~\ref{fig:hm} the possible shapes for $T'/G$, and a path in $\FF_4$ between the two one-edge collapses of $T'$.

Let now $\theta':\FF_4\to \FF_1$ be a map assigning to a deformation space of free splittings, any free splitting in this deformation space. It is easy to check that both $\theta$ and $\theta'$ are coarsely Lipschitz, and quasi-inverse of each other. Therefore $\theta$ defines a quasi-isometry. We leave the details to the reader.
\end{proof}

\begin{rk}
We comment on low complexity cases. 
\begin{itemize}
\item When $\xi(G,\calf)\le (2,0)$, then $G=\{1\}$, $G=G_1$, $G=\mathbb{Z}$ or $G=G_1\ast G_2$, and in all these cases $\FF_3$ and $\FF_4$ are empty. 
\item When $\xi(G,\calf)=(2,1)$, then $G=G_1\ast\mathbb{Z}$ and $\FF_3$ and $\FF_4$ are totally disconnected (all proper free factors are cyclic).
\item When $\xi(G,\calf)=(3,0)$, then $G=G_1\ast G_2\ast G_3$ and every proper free factor is of the form $\langle G_i, G_j^a\rangle$, so $\FF_3$ and $\FF_4$ are totally disconnected.
\item When $\xi(G,\calf)=(2,2)$, then $G=F_2$ and all proper free factors are cyclic, so $\FF_3$ is totally disconnected; $\FF_4$ is a subdivision of the Farey graph.
\item When $\xi(G,\calf)=(3,1)$, then $G=G_1\ast G_2\ast\mathbb{Z}$, so all proper free factors are of the form $\langle G_1,G_2^a\rangle$ or $\langle G_1,t\rangle$ or $\langle t\rangle$, and $\FF_3$ is disconnected. 
\end{itemize}
\end{rk}

\subsubsection{Hyperbolicity}

Hyperbolicity of the free factor graph was proved by Bestvina--Feighn \cite{BF14} in the case of free groups, and by Handel--Mosher in \cite[Theorem 1.4]{HM14} in the relative setting for all nonsporadic cases except $(G,\calf)=G_1*G_2*G_3$.
In order to describe $\partial_\infty \FF$, we will need to know that $\FF$ is also an electrification of $\mathcal{Z}S$,
i.e.\ to construct the coarsely alignment preserving map $j:\ZS\ra \FF$ mentioned in the introduction.
This is done by applying Kapovich--Rafi's criterion, and proves simultaneously the hyperbolicity of $\FF$, including the case 
where $(G,\calf)=G_1*G_2*G_3$.

\begin{prop}\label{fz-ff}
The inclusion map
$j:\ZS\to \FF_2$ is a coarsely surjective, coarsely alignment-preserving map from $\ZS$ to $\FF_2$. 
\\ In particular $\FF$ is Gromov hyperbolic, and optimal folding paths between any two trees with trivial edge stabilizers are unparametrized quasi-geodesics with constants that depend only on $(G,\calf)$ (in particular they are uniformly close to geodesics).
\end{prop}

\begin{proof}
The map $j$ is clearly surjective on vertices. 
In view of Proposition~\ref{KR}, it is thus enough to show that if $T$ and $T'$ are two 
$\Z$-splittings at distance $1$ from each other in $\FF_2$, then there exists a geodesic from $T$ to $T'$ in $\ZS$, whose image in $\FF_2$ has bounded diameter. 
If $T$ and $T'$ are compatible, then $T$ and $T'$ are at distance $1$ in $\ZS$, and we are done. 
We can thus assume that $T$ and $T'$ share a common nonperipheral simple elliptic element $a\in G$. Denoting by $A\subseteq G$ the smallest $(G,\calf)$-free factor that contains $a$, Lemma~\ref{lem_depliage} says that we can find Grushko $(G,\calf\cup\{[A]\})$-trees $S,S'$ at bounded distance from $T,T'$ in $\ZS$. 
Consider an optimal folding path $\gamma$ from $S$ to $S'$ among Grushko $(G,\calf\cup\{[A]\})$-trees. In particular, $a$ is elliptic in all intermediate trees,
so the image of $\gamma$ in $FF_2$ has diameter at most $1$.
Now $\gamma$ is an (unparameterized) quasi-geodesic in $\ZS$ with uniform constants \cite[Theorem~3.3]{Hor14-2},
so it lies at bounded distance from any geodesic from $T$ to $T'$.
\end{proof}

\subsection{The $\calz$-factor graph}\label{sec-z}

In the present section, we define the $\calz$-factor graph, which turns out to be quasi-isometric in the context of free groups to a version of Kapovich--Lustig's intersection graph \cite{KL09} or Dowdall--Taylor's co-surface graph \cite{DT16}.

\subsubsection{Definition and various models}

\begin{de}[\textbf{\emph{$\calz$-factor graph}}]
The \emph{$\calz$-factor graph} $\ZF$ is the graph whose vertices are the $\calz$-splittings of $(G,\calf)$, in which two splittings are joined by an edge if they are compatible or have a common nonperipheral elliptic element.  
\end{de}

Notice that the only difference with the version $\FF_2$ of the relative free factor graph is that the common elliptic elements defining the edges of $\ZF$ are no longer required to be simple. 

Here is an example of trees that are far when viewed in $\FF_2$ but are close in $\ZF$ in the context of free groups.
Let $\Sigma$ be a surface with boundary, and let $\grp{b}$ be the fundamental group of a boundary component. Let $c,c'$ be two simple closed curves in $\Sigma$. The dual splittings $T_c,T_{c'}$ of $\pi_1(\Sigma)$ are cyclic splittings that are at distance at most $1$ in $\ZF$ because $b$ is elliptic in both $T_c$ and $T_{c'}$. If $\Sigma$ has several boundary components, then $b$ is simple so in fact $T_c$ and $T_{c'}$ are at distance at most $1$ in $\FF_2$. 
On the other hand, if $\Sigma$ has a single boundary component, then $c,c'$ can be chosen so that $T_c$ and
$T_{c'}$ are arbitrarily far in $\FF_2$ (for example, one can take $c'$ to be the image of $c$ under a high power of a pseudo-Anosov diffeomorphism $\Phi$ of $\Sigma$, and use the fact that $\Phi$ is a loxodromic isometry of $\FF$ \cite{BF14}).

\begin{figure}[htb]
  \centering
\includegraphics[scale=0.8]{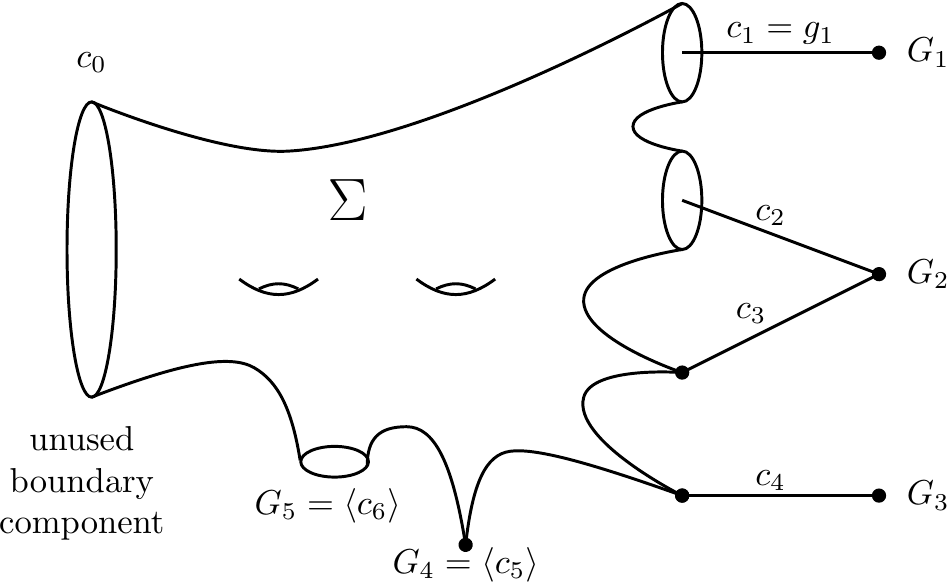}
  \caption{A geometric decomposition.} 
  \label{fig-arat-surf}
\end{figure}

We will actually show below (Theorem~\ref{zfq}) that in general, if $T,T'$ are at distance $1$ from each other in $\ZF$ but far enough in $\FF_2$, then they come from a similar situation involving a surface; this motivates the following definitions.

A \emph{$QH$ vertex} in a $(G,\calf)$-splitting $S$ is a vertex $v$ whose vertex group $G_v$
is identified with the fundamental group of $\Sigma$, a compact, connected (possibly non-orientable) $2$-orbifold with conical singularities,
and such that all incident edge groups and all peripheral subgroups contained in $G_v$ are conjugate into a boundary or conical subgroup of $\pi_1(\Sigma)$. Note that in \cite{GL16}, the definition of QH is more general as it allows orbifolds with mirrors and a possible fiber.

A \emph{geometric decomposition} of $(G,\calf)$ (see Figure \ref{fig-arat-surf}) is a (maybe trivial) splitting of $(G,\calf)$ with a QH vertex group $G_v=\pi_1(\Sigma)$, such that the stabilizer of every edge and of every vertex outside of the orbit of $v$ is peripheral,
and all edge stabilizers are nontrivial and cyclic (possibly finite). 
We call $\Sigma$ the \emph{underlying orbifold} of the decomposition. 
 
Every conical group $\grp{c}$ of $\Sigma$ is necessarily peripheral; it might happen that no incident
edge group is conjugate to $\grp{c}$, in which case $\grp{c}$ is a group in $\calf$.
Given $\grp{c}$ the fundamental group of a boundary component of $\Sigma$, it might happen that no
incident edge group is conjugate in $\grp{c}$, in which case $\grp{c}$ might be peripheral or not.
When $\grp{c}$ is not peripheral, we say that the corresponding boundary component is \emph{unused}.

\begin{rk}
In a geometric decomposition, it is allowed to have several edge groups conjugate into the same boundary
group, or to have an edge group properly contained in a boundary subgroup.
Replacing $S$ by its tree of cylinders $S_c$ for the co-peripheral equivalence relation (i.e.\ $G_e\sim G_{e'}$
if $\grp{G_e,G_{e'}}$ is peripheral), one gets another geometric decomposition where these peculiarities do not appear.
One can then describe $S_c$ as follows:
one of the vertex groups of $S_c$ is the fundamental group of the orbifold $\Sigma$, 
and the other vertex groups are a subcollection of the peripheral subgroups $G_1,\dots,G_k$. 
Then we add some edges amalgamating a boundary or conical subgroup of $G_v$ to a subgroup of some $G_i$.
Choices are made in such a way that for each conical group and each boundary group, there is at most one edge carrying a conjugate of this group.
The peripheral subgroups $G_j$ that do not appear as vertex groups are conjugate to a boundary or conical subgroup of $\Sigma$ which  does not appear as an incident edge group.
\end{rk}

Since $\pi_1(\Sigma)$ is freely indecomposable relative to its boundary subgroups, there has to be at least one unused boundary component.
We note that if $\Sigma$ has several unused boundary components, then the fundamental group of each of them is simple (as can be seen by looking at the free splitting dual to a properly embedded arc with endpoints in another boundary component).

\begin{de}[\textbf{\emph{Quadratic element}}]\label{def-quadratic}
An element $g\in G$ is \emph{quadratic} if it is non-simple, and occurs as a generator of the fundamental group of the (single) unused boundary component in some geometric decomposition of $(G,\calf)$.
\end{de}

See Lemma \ref{quadratic} below for an equivalent characterization of quadratic elements which justifies the terminology.

In the context of free groups, another definition of the $\calz$-factor graph was introduced by Mann \cite{Man14} (where it was called \emph{intersection graph},
and which turns out to be quasi-isometric to Dowdall-Taylor's \emph{co-surface graph} \cite[Proposition~4.1]{DT16}). In his definition, Mann only joins two splittings by an edge if they have a common simple or quadratic elliptic element, as follows. 

\begin{de}[\textbf{\emph{$\calz$-factor graph, quadratic version}}]
We let $\ZF_{q}$ be the graph whose vertices are the $\calz$-splittings of $(G,\calf)$, in which two splittings are joined by an edge if they are compatible or have a common nonperipheral elliptic element \textbf{which is simple or quadratic}.  
\end{de}

There is an obvious inclusion map $i:\ZF_q\to \ZF$. It turns out that this inclusion is actually an  isomorphism, and the two graphs are the same. 
This is a consequence of the following theorem.

\begin{theo}\label{zfq}
Let $T$ and $T'$ be two noncompatible $\calz$-splittings that have a common nonperipheral elliptic element $g$ which is not a proper power.
\\ Then $T$ and $T'$ have a common nonperipheral elliptic element which is simple or quadratic.
\\ More precisely, either $T$ and $T'$ share a nonperipheral simple elliptic element, or the cyclic JSJ decomposition of $G$ relative to $\calf\cup\{\grp{g}\}$
is a geometric decomposition in which $\grp{g}$ is conjugate to the fundamental group of an unused boundary component (in particular $g$ is quadratic). 
\end{theo}

\begin{cor}
The inclusion $i$ is a graph isomorphism $\ZF_q\simeq \ZF$.
\qed 
\end{cor}

\begin{proof}[Proof of Theorem~\ref{zfq}]
Assume that $T$ and $T'$ do not have any simple nonperipheral elliptic element in common, and let $g$ be a common nonperipheral elliptic element, which is not a proper power.
We will prove that $g$ is quadratic.

Since $g$ is not simple, $G$ is freely indecomposable relative to $\calf\cup\{\langle g\rangle\}$. 
Beware that $\calf\cup\{\langle g\rangle\}$ is not a free factor system; so below, when we talk about peripheral subgroups, it is relative to the free factor system $\calf$.
Let $T_J$ be a JSJ splitting of $G$ relative to $\calf\cup\{\langle g\rangle\}$ over the class of cyclic (finite or infinite) subgroups.
Its flexible vertex groups are $QH$ (\cite[Theorem 6.5]{GL16}).

Edge stabilizers in $T_J$ are elliptic in all splittings of $(G,\calf\cup\{\langle g\rangle\})$ over cyclic subgroups, in particular they are elliptic in $T$ and $T'$. By Lemma \ref{lem_tirer},  
edge stabilizers in a cyclic splitting of $(G,\calf)$ are simple (possibly peripheral). Thus, if $T_J$ had an edge with nonperipheral stabilizer, then $T$ and $T'$ would share a simple nonperipheral elliptic element, a contradiction. Hence all edge stabilizers in $T_J$ are peripheral (and nontrivial because $G$ is freely indecomposable relative to $\calf\cup\{\langle g\rangle\}$). 

We claim that $T_J$ contains exactly one orbit of vertices with nonperipheral stabilizer. 
To prove that there is at least one, consider $\hat T$ obtained from $T_J$ by blowing up each vertex $v$ with nonperipheral stabilizer using
a Kurosh decomposition of $G_v$ (this is possible because all edge stabilizers are peripheral).
Then $\hat T$ dominates any Grushko $(G,\calf)$-tree. 
In particular, $\hat T\neq T_J$, which proves that $T_J$ contains at least one vertex with nonperipheral stabilizer.
Now assume that there exist two vertices $v,w$ of $T_J$ not in the same orbit such that $G_v$ and $G_w$ are both nonperipheral. Then up to exchanging $v$ and $w$, we can assume that $g$ fixes a vertex not in the orbit of $v$. The group $G_v$ has a nontrivial Kurosh decomposition, and this can be used to blowup $T_J$ at $G_v$ into a free splitting of $(G,\calf)$ relative to $g$. This shows that $g$ is simple, a contradiction.

We now claim that the unique vertex $v\in T_J/G$ with nonperipheral vertex group is a QH vertex. If not, $G_v$ is universally elliptic with respect to all cyclic splittings of $G$ relative to $\calf\cup\{\langle g\rangle\}$, so every element of $G_v$ is elliptic in both $T$ and $T'$. Blowing up $v$ using a Kurosh decomposition of $G_v$ yields a tree $\hat T$ dominating a Grushko $(G,\calf)$-tree;
since $(G,\calf)$ is nonsporadic, this Kurosh decomposition cannot be of the form $A\ast B$, so $G_v$ contains a simple element.
It follows that $T$ and $T'$ share a nonperipheral simple elliptic element, a contradiction.
\end{proof}

\subsubsection{Hyperbolicity}

The proof of the following theorem is due to Mann \cite{Man14} in the context of free groups. We extend it to the case of free products.

\begin{theo}\label{zf-hyp}
There exists a coarsely surjective, coarsely alignment-preserving map from $\FF$ to $\ZF$.
\\ In particular $\ZF$ is Gromov hyperbolic. 
\end{theo}

Before proving the theorem, we give an alternative definition of quadratic elements (that justifies the name).

Note that a nonperipheral element $g\in G$ is conjugate to its inverse if and only if it is contained in an infinite dihedral group.
In particular, a quadratic element is not conjugate to its inverse.

\begin{lemma}\label{quadratic} Let $g\in G$ be a nonperipheral element which is not simple and not conjugate to its inverse.
\\ Then $g$ is quadratic if and only if 
there exists a Grushko tree $R$ such that some fundamental domain for the axis of $g$ in $R$ intersects each orbit of edges exactly twice (regardless of orientation).
\\ More precisely, if $R$ is a Grushko tree, then the following are equivalent:
\begin{enumerate}\renewcommand{\theenumi}{(\roman{enumi})}\renewcommand{\labelenumi}{\upshape\theenumi}
\item  there exists a geometric decomposition of $(G,\calf)$ with underlying orbifold $\Sigma$ such that $g$ generates the fundamental group of the unused boundary component $b$ of $\Sigma$, and $R$ is dual to a collection of disjoint properly embedded arcs on $\Sigma$ with endpoints in $b$;
\item some fundamental domain for the axis of $g$ in $R$ intersects each orbit of edges exactly twice.
\end{enumerate}
\end{lemma}

We will say that a nonsimple element $g$ which is not conjugate to its inverse is quadratic \emph{in $R$} 
whenever $R$ satisfies (ii). 

\begin{proof}
Assume first that $\grp{g}$ is conjugate to the fundamental group of an unused boundary component $b$ of the underlying orbifold $\Sigma$ of some geometric decomposition of $(G,\calf)$.
Let $\cali$ be a maximal collection of disjoint, properly embedded, non-parallel arcs in $\Sigma$ with endpoints in $b$. Then the splitting $R$ dual to $\cali$ is a Grushko splitting, and some fundamental domain for the axis of $g$ in $R$ intersects each orbit of edges exactly twice. 

Conversely, assume that $g$ satisfies (ii), 
and let $R$ be a Grushko tree such that the axis of $g$ in $R$ intersects each orbit of edges exactly twice. 
If $g=h^2$ is the square of some element $h$,
then some fundamental domain for the axis of $h$ in $R$ intersects each orbit of edges once, so 
$h$ is simple, and so is $g$. Therefore $g$ is not a proper power. 
Since $g$ is not contained in an infinite dihedral group, $\grp{g}$ is the full stabilizer of its axis $A_g$.
Let $B=[0,1]\times \bbR$ be a bi-infinite band, and let $X$ be the square complex obtained from $R$ by gluing a copy of $B$ on each translate of $A_g$ along $\{0\}\times \bbR$. 
We call the \emph{boundary} $\partial X$ of $X$ the union of copies of $\{1\}\times \bbR$ in $X$.
Our hypothesis on $g$ shows that $X$ is a pseudo-surface with boundary $\partial X$: every edge is contained in exactly two squares, 
except the edges in $\partial X$.
Since $\grp{g}$ is the stabilizer of $A_g$ and acts freely on $A_g$, every edge and square of $X$ has trivial stabilizer.
Being a $1$-manifold, each connected component $l$ of the link in $X$ of a vertex $x\in X\setminus \partial X$ is a circle or a line,
and the stabilizer $G_l\subseteq G_x$ of $l$ acts freely on $l$.
As there are only finitely many  $G_x$-orbits of edges incident on a vertex $x\in R$, the group $G_l$ acts cocompactly on $l$. 
Therefore $G_l$ is cyclic (finite or infinite) and $l/G_l$ is a circle.

If there exists some vertex $x\in X\setminus\partial X$ with trivial stabilizer
and whose link is not connected (hence a finite union of circles), then $g$ is simple. Indeed, the link of $x$ is naturally identified with the Whitehead graph of $g$ at $x$ in $R$, so by \cite[Proposition~5.1]{GH15-1},
$g$ is simple (alternatively,
blowing up the orbit of $x$ in $X$, one can directly construct a $(G,\calf)$-free splitting in which $g$ is elliptic).

Choose $\eps>0$ small enough, and for each $x\in X\setminus\partial X$ with nontrivial stabilizer, let $S_x$ be the sphere of radius $\eps$ around $x$.
Each connected component of $S_x$ separates $X$ (because $X$ is simply connected). Let $T$ be the simplicial $(G,\calf)$-tree having one vertex $v_C$ for each connected component $C$ of the complement of the union of the spheres $S_x$, two vertices $v_C,v_{C'}$ being joined by an edge if the closures of $C$ and $C'$ intersect.
Either $C$ is the ball of radius $\eps$ around some vertex $x$ with non-trivial stabilizer, in which case the stabilizer of the vertex $v_C$ is the corresponding peripheral group.
Otherwise, $C$ is a surface with boundary, and each boundary component of $C$ is of one of the following types: either it is 
a line in $\partial X$,
or else it is a connected component of some $S_x$, which can be either a circle or a line.
Let $C'$ be the union of $C$ together with the disks bounded by the circles in $\partial C$.
Then $C'$ is a simply connected surface all whose boundary components are lines, and on which $G_C$ acts properly
(the only non-trivial point stabilizers are the
centers of added disks).
Thus, $G_C$ is the fundamental group of the orbifold $C'/G_C$ with conical singularities, 
and $T/G$ is a geometric decomposition of $G$, and $\grp{g}$ is the fundamental group of an unused boundary component of $C'/G_C$.

We finally check that $R$ is dual to a system of arcs $\cali$ in $C'/G_C$. Let $\tilde \cali_B$ be the union of segments
of $B=[0,1]\times \bbR$ of the form $[0,1]\times \{m\}$ such that $(0,m)$ is glued to the midpoint of an edge in $R$.
Let $\tilde \cali$ be the union of translates of $\tilde \cali_B$. 
This is a $G$-invariant family of disjoint arcs, each of which being the union of two translates of arcs in $\tilde \cali_B$.
The tree dual to this family of arcs is isomorphic to $R$, and the image of $\Tilde \cali$ in $C'/G_C$
is a finite disjoint union of properly embedded arcs as required.
\end{proof}

\begin{cor}\label{cor-quadratic} Let $(G,\calf)$ be a nonsporadic free product, let $R$ be a Grushko tree, and let $g$ be a nonsimple element which is quadratic in $R$.
Then there exists a $\calz$-splitting $S$ of $(G,\calf)$ compatible with a collapse of $R$ such that $g$ is elliptic in $S$. 
\\ In particular, we have $d_{\ZF_q}(S,R)\leq 2$.
\end{cor}

\begin{figure}[htb]
\centering
\includegraphics{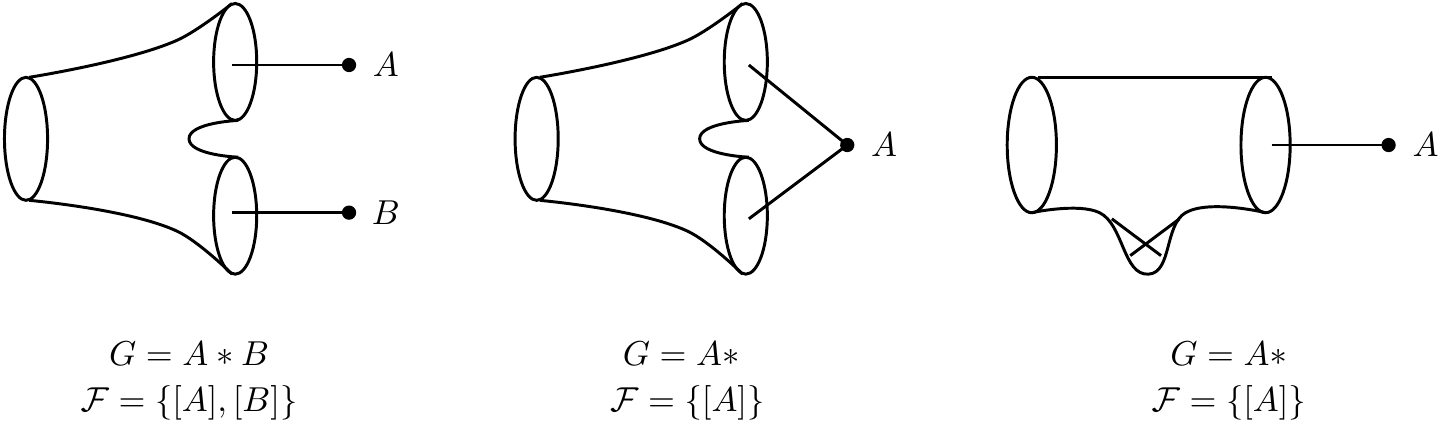}
\caption{Surfaces of small complexity lead to sporadic cases.}
\label{fig-sporadic}
\end{figure}

\begin{proof}
By Lemma \ref{quadratic}, the tree $R$ is dual to a system of arcs on an orbifold $\Sigma$ in a geometric decomposition of $(G,\calf)$. 
As $g$ is nonsimple, the orbifold $\Sigma$ has a single unused boundary component.
Let $R'$ be the collapse of $R$ dual to a single arc $\gamma$ of this collection.
Unless when $\Sigma$ is a sphere with at most 3 punctures or cone points, or when $\Sigma$ is
a projective plane with at most 2 punctures or cone points, then one can find a simple closed curve $c$ disjoint from $\gamma$. The tree $S$ dual to $c$ satisfies the lemma. In the remaining cases, one checks that $(G,\calf)$ is sporadic (see Figure~\ref{fig-sporadic}).
\end{proof}

\begin{proof}[Proof of Theorem~\ref{zf-hyp}]
We will work with the models $\FF_2$ and $\ZF_{q}$. In view of Proposition \ref{KR}, it is enough to prove that if $S,S'$ are two $\calz$-splittings of $(G,\calf)$ at distance $1$ from one another in $\ZF_{q}$, then there exists a geodesic from $S$ to $S'$ in $\FF_2$ whose image in $\ZF_{q}$ has bounded diameter. 
We can assume that $d_{\FF_2}(S,S')\ge 2$, as otherwise this is obvious. In other words $S$ and $S'$ are not compatible, and do not share any nonperipheral simple elliptic element, but they
share a nonperipheral elliptic element $g$. Without loss of generality, $g$ is not a proper power.
By Theorem~\ref{zfq},
the cyclic JSJ decomposition $T_J$ of $G$ relative to $\calf\cup\{\grp{g}\}$
is a geometric decomposition of $(G,\calf)$ in which $\grp{g}$ is conjugate to the fundamental group
of an unused boundary component. In particular, $g$ is not conjugate to its inverse.
The splittings $S$ and $S'$ are dominated by splittings $\Tilde S,\Tilde S'$ dual to curves on the underlying orbifold $\Sigma$. Since edge stabilizers of $\Tilde S$ fix an edge in $S$ and are simple,
$S$ and $\tilde S$ are at distance at most 1 in $\FF_2$. Similarly, $d_{FF_2}(S',\tilde S')\leq 1$.
We can find Grushko trees $R,R'$ which are dual to maximal collections of arcs on $\Sigma$ with all endpoints on the unused boundary curve, such that $R$ (resp.\ $R'$) 
has a collapse that is compatible with $\tilde S$ (resp.\ $\tilde S'$). In view of Proposition~\ref{fz-ff}, it is enough to show that there exists an optimal folding path from $R$ to $R'$ whose image in $\ZF_q$ has bounded diameter.

Let $g$ be a quadratic element that generates the fundamental group of the unused boundary curve of $\Sigma$. Then some fundamental domain of the axis of $g$ crosses every orbit of edges exactly twice in both $R$ and $R'$.

Let $f:R\ra R'$ be an optimal map, and up to subdividing, assume without loss of generality that it sends edge to edge and vertex to vertex. We assign length 1 to each edge. If $f$ is not an isomorphism, then $\Vol(R/G)>\Vol(R'/G)$; since $||g||_R=2\Vol(R/G)$ and $||g||_{R'}=2\Vol(R'/G)$, we get that $||g||_{R}>||g||_{R'}$.
It follows that there are two adjacent edges $e,e'$ in $A_g$ that are folded by $f$.

Let $R_1$ be the Grushko $(G,\calf)$-tree obtained from $R$ by folding these two edges. Let $I_1\subset R_1$ 
be a fundamental domain of the axis of $g$ with endpoints at vertices of $R_1$.
One has $||g||_{R_1}\leq ||g||_{R}-2$, and since $g$ is not simple, $I_1$ has to contain at least two edges in 
each orbit, so $||g||_{R_1}\geq 2\Vol(R_1/G)=||g||_{R}-2$. It follows that $||g||_{R_1}=||g||_{R}-2=2\Vol(R_1/G)$
so that $I_1$ crosses every orbit of edges exactly twice.
Arguing by induction, one constructs an optimal folding path $R=R_0$, $R_1$,\dots, $R_n=R'$ such that
$g$ is quadratic in $R_i$ for all $i$. By Corollary~\ref{cor-quadratic}, the diameter in $\ZF_q$ of $\{R_0,\dots, R_n\}$ is
at most 4.
\end{proof}

\subsection{The $\ZRC$-factor graph}\label{sec-zmax}

A cyclic subgroup $H$ of a free product $(G,\calf)$ is $\ZRC$ if it is 
nonperipheral and root-closed (maybe trivial). A \emph{$\ZRC$-splitting} of $(G,\calf)$ is a splitting of $(G,\calf)$ over $\ZRC$ groups.

In \cite{Hor14-2}, the second author proved that the graph of $\ZRC$-splittings (called $\Zmax$-splittings in \cite{Hor14-2})
is not quasi-isometric to the graph of $\Z$-splittings.
In this section, we show on the contrary that the graph of $\ZRC$-factors is quasi-isometric to the graph of $\Z$-factors.
This section will not be used in the rest of the paper.

\begin{de}[\textbf{\emph{$\ZRC$-factor graph}}]
The \emph{$\ZRC$-factor graph} $\ZRCF$ is the graph whose vertices are the  $\ZRC$ splittings of $(G,\calf)$,
and in which two splittings are joined by an edge if they are compatible or have a common nonperipheral elliptic element.  
\end{de}

There is a natural inclusion map $i:\ZRCF\ra \ZF$ which is clearly $1$-Lipschitz.

\begin{prop}\label{z-zmax} 
The map $i$ is a quasi-isometry between $\ZRCF$ and $ \ZF$.
\end{prop}

\begin{proof}[Proof of Proposition~\ref{z-zmax}]
  If $(G,\calf)$ is sporadic then $\ZF$ and $\ZRCF$ are bounded because $\FS$ is and every $\calz$-splitting
is compatible with a free splitting (Lemma~\ref{lem_tirer}). The graph $\ZRCF$ is also bounded if $(G,\calf)=(F_2,\es)$. Indeed, taking $F_2=\grp{a,b}$, every $\ZRC$-splitting $S$ is at distance at most 2 from a $\ZRC$-splitting $S'$ whose edge stabilizers are all non-trivial; but the commutator $[a,b]$ is elliptic in all such $S'$, so $\ZRCF$ has finite diameter.
Since $i$ is almost surjective (see below), $\ZF$ is also  bounded.

We now assume that  $(G,\calf)$ is nonsporadic and $G\neq F_2$.
Given $T\in\ZF$, we define 
$\tilde\theta(T)\subset \ZRCF$ as follows: 
if edge stabilizers of $T$ are trivial, $\tilde \theta(T)$ is the set of trees $S\in \ZRCF$ which are compatible with $T$;
otherwise  
$\tilde\theta(T)$ is the set of trees $S\in \ZRCF$ such that every edge group of $T$ is elliptic in $S$.
We note that in both cases, $\tilde \theta(T)$ contains all $\ZRC$ splittings compatible with $T$. The set $\tilde\theta(T)$  has diameter at most $2$ and is nonempty by Lemma \ref{lem_tirer}. We define $\theta(T)$ by choosing some element in $\tilde \theta(T)$.
Since $i\circ \theta(T)$ is at distance $1$ from $T$ in $\ZF$, and since $S\in \tilde\theta( i(S))$, $\theta$ is a quasi-inverse of $i$. 

It suffices to prove that if $T_1,T_2\in\ZF$ are at distance $1$, then $\tilde\theta(T_1)$ and $\tilde\theta(T_2)$ are at bounded distance from each other. 
If $T_1,T_2\in\ZF$ have a common refinement $T\in\ZF$, then $\tilde\theta(T)$ is nonempty and contained in $\tilde\theta(T_1)\cap\tilde\theta(T_2)$ and the result is clear.
Otherwise, there exists a nonperipheral element $g\in G$ which is elliptic in $T_1$ and $T_2$.
If $g$ is simple, then there is a free factor system $\calf_g$ in which 
$g$ and all elements of $\calf$ are peripheral, and such that $T_1$ and $T_2$ are $(G,\calf_g)$-trees.
By Lemma~\ref{lem_tirer},  
there exists a $(G,\calf_g)$-free splitting $S_i$ compatible with $T_i$;
then $S_i\in\tilde\theta(T_i)$ and $S_1$ and $S_2$ are at distance at most $1$ since $g$ is elliptic in both of them.

This shows that we can assume that there is no nontrivial $(G,\calf)$-free splitting in which $g$ is elliptic, and in particular,
all edge stabilizers of $T_i$ are nontrivial.
Replacing $T_i$ by a collapse $T'_i$, and using that $\tilde\theta(T'_i)\supseteq \tilde\theta(T_i)$, we may assume that $T_i$ has a single orbit of edges.
We denote by $a_i$ a generator of an edge stabilizer of $T_i$.
We consider a free splitting $S_i$ compatible with $T_i$; in particular $a_i$ is elliptic in $S_i$ and $S_i\in\tilde\theta(T_i)$.

We first assume that $a_1$ is elliptic in $T_2$. Since edge stabilizers of $S_1$ are trivial, and since $a_1$ is elliptic in both $S_1$ and $T_2$, we can consider a blowup $\hat S_1$ of $S_1$ dominating $T_2$ in which $a_1$ is elliptic
(see \cite[Proposition~2.2]{GL16} for instance). Let $U_1,\dots,U_n=T_2$ be a folding path between a collapse $U_1$ of $\hat S_1$ and $T_2$,
and let $t< n$ be the last time for which $U_i$ has no edge stabilizer commensurable with $a_2$.
Then edge stabilizers of $U_t$ are trivial because any edge in a subdivision of $U_t$
that is mapped to an edge of $T_2$ has a stabilizer contained in a conjugate of $\grp{a_2}$.
Moreover, $a_2$ is elliptic in $U_t$ because  $U_t$ is compatible with $U_{t+1}$ and some power of $a_2$ fixes an edge in $U_{t+1}$. Therefore $U_t$ is at distance at most $1$ from $S_2$.
Since $a_1$ is elliptic in $U_t$ (in fact it is elliptic in all trees $U_1,\dots,U_n$), we deduce that $U_t$ is also at distance at most $1$ from $S_1$. This shows that $\tilde\theta(T_1)$ and $\tilde\theta(T_2)$ are at bounded distance from each other.

By symmetry, if $a_2$ is elliptic in $T_1$, the sets $\tilde\theta(T_1)$ and $\tilde\theta(T_2)$ are at bounded distance from each other.

We now assume that $a_1$ is hyperbolic in $T_2$, and $a_2$ is hyperbolic in $T_1$.
We are going to construct $R_1,R_2\in\ZRCF$ such that $a_i$ and $g$ are elliptic in $R_i$.
This will conclude because $S_1,R_1,R_2,S_2$ is a path of length  at most 3 between $S_1\in\tilde\theta(T_1)$ and $S_2\in\tilde\theta(T_2)$.

The splittings $T_1,T_2$ are cyclic splittings relative to $\calh=\calf\cup\{\langle g\rangle\}$.
Since $G$ has no free splitting relative to $\calh$, one can construct
the regular neighbourhood $R$ of $T_1,T_2$ from Fujiwara-Papasoglu's core of $T_1\times T_2$ (see \cite[Definition~6.23, Lemma~6.22 and Remark~6.10]{GL16}).
We will use the following features of $R$ (see \cite[Proposition~6.25]{GL16}, where $\cala$ is the class of cyclic groups):
its edge groups are cyclic (maybe peripheral),
$R/G$ has a vertex $v$ such that $G_v$ is QH (it cannot be virtually $\bbZ^2$ because it is not peripheral as it acts hyperbolically on $T_1$ and $T_2$)
with an underlying conical orbifold $\Sigma$,
and $T_i$ is dual to some essential simple closed geodesic $\gamma_i$ in $\Sigma$.
Every edge of $R/G$ has exactly one of its endpoints at $v$.
Note that $a_1,a_2,g$ are elliptic in $R$, and that either $g$ is conjugate in $G_w$ with $w\neq v$, or $g$ is 
conjugate to a boundary subgroup of $G_v=\pi_1(\Sigma)$.

Given a nonperipheral cyclic group $G_e$, denote by $\hat G_e$ the unique $\ZRC$ subgroup containing $G_e$ with finite index.
Let $R^\RC$ be the (maybe trivial) tree with $\ZRC$ edge stabilizers obtained from $R$ by identifying 
every edge $e$ whose stabilizer is nonperipheral with all its translates $ge$ for $g\in \hat G_e$,
and by passing to the minimal subtree (see \cite[Lemma~9.27]{GL16} in the case of hyperbolic groups).
If $R^\RC$ has a nonperipheral edge group, then the corresponding one-edge splitting $\bar R$ is a tree in $\ZRCF$
in which $a_1,a_2,g$ are all elliptic, and we can take $R_1=R_2=\ol R$.

Thus we can assume that all edges of $R^\RC$ are peripheral, which means that the only nonperipheral edge groups of $R/G$ are terminal edges
$e$ joining $v$ to a vertex with cyclic stabilizer.

If the geodesic $\gamma_i$ is two-sided, then the splitting $T_i$ is root-closed, and we can take $R_i=T_i$.
Otherwise, $\gamma_i$ has a tubular neighbourhood homeomorphic to a Mobius band $M$, and consider $\Sigma'_i:=\Sigma\setminus \rond M$
(still a hyperbolic orbifold).
If $\Sigma'_i$ contains a two-sided closed geodesic, then one takes for $R_i$ the splitting dual to this geodesic:
it is root closed, and $a_i$ and $g$ are elliptic in $R_i$.
If $\Sigma'_i$ contains no two-sided closed geodesic, then there are few possibilities for $\Sigma$.
By \cite[Section~5.1.4]{GL16}, the only possibilities for $\Sigma'_i$ are
a sphere with at most 3 boundary components or conical points,
or a projective plane with at most 2 boundary components or conical points.
This implies that $\Sigma$ is either a projective plane with at most 2 boundary components or conical points,
or a Klein bottle with exactly one boundary component.
At least one boundary component has to be nonperipheral and not amalgamated to a (terminal) cyclic vertex group
because otherwise $G$ would be one-ended relative to $\calf$.
This forces $(G,\calf)=(F_2,\es)$ in the second case
and 
$(G,\calf)=(F_2,\es)$ or $(G,\calf)=(G_1*\bbZ,\{[G_1]\})$ in the first case.
\end{proof}

\section{Gromov boundaries}\label{sec-bdy}

\subsection{The Gromov boundary of the $\calz$-splitting graph: review}

The Gromov boundary of $\ZS$ was described in \cite{Hor14-2}, as follows. 
An $\bbR$-tree $T\in\overline{\mathcal{O}}$ is \emph{$\mathcal{Z}$-compatible} if it is compatible with some $\mathcal{Z}$-splitting of $(G,\mathcal{F})$. It is \emph{$\mathcal{Z}$-averse} if it is not compatible with any $\mathcal{Z}$-compatible tree $T'\in\overline{\mathcal{O}}$. We denote by $\ZA$ the subspace of $\overline{\mathcal{O}}$ consisting of $\mathcal{Z}$-averse trees. Two trees $T,T'\in\ZA$ are \emph{$\ZA$-equivalent}, which we denote by $T\simZA T'$, if they are both compatible with a common tree in $\overline{\mathcal{O}}$ (although not obvious, it follows from \cite[Theorem~5.1]{Hor14-2} that this is an equivalence relation on $\ZA$). There is an $\text{Out}(G,\calf)$-equivariant map $\psi_{\ZS}:\mathcal{O}\to \ZS$ given by forgetting the metric.

\begin{theo} (\cite{Hor14-2}) \label{boundary-fz}
There is a unique $\text{Out}(G,\mathcal{F})$-equivariant homeomorphism $$\partial{\psi}_{\ZS}:\ZA/\!\raisebox{-.2em}{$\simZA$}\to\partial_{\infty} \ZS,$$ so that for all $T\in\ZA$, and all sequences $(T_n)_{n\in\mathbb{N}}\in \mathcal{O}^{\mathbb{N}}$ converging to $T$, the sequence $(\psi_{\ZS}(T_n))_{n\in\mathbb{N}}$ converges to $\partial{\psi}_{\ZS}(T)$. 
\\ Moreover, if $(T_n)_{n\in\mathbb{N}}\in\mathcal{O}^{\mathbb{N}}$ converges to a tree $T\in\ol\calo\setminus\ZA$, then $\psi_{\ZS}(T_n)$ has no accumulation point in $\partial_\infty \ZS$. 
\end{theo}

\subsection{The Gromov boundary of the free factor graph}\label{sec-6}

We will assume throughout the section that $(G,\calf)$ is nonsporadic, as otherwise the free factor graph $\FF$ has bounded diameter. Theorems \ref{intro-bdy-ff} and \ref{intro-bdy-I} hold trivially in this case since $(G,\calf)$ has no arational tree.

Our first main result gives a description of the Gromov boundary of the free factor graph in terms of arational trees. This description is due to Bestvina--Reynolds \cite{BR13} and Hamenstädt \cite{Ham13} for free groups with empty peripheral structure
(arational trees were introduced by Reynolds in \cite{Rey12} in this context). 

\begin{de}[\textbf{Arational tree}]
A $(G,\mathcal{F})$-tree $T\in\overline{\mathcal{O}}$ is \emph{arational} if $T\in\partial \mathcal{O}$ and for every proper $(G,\mathcal{F})$-free factor $H\subset G$, the factor $H$ is not elliptic in $T$, and the action of $H$ on its minimal subtree $T_H$ is simplicial and relatively free. 
\end{de}   

Equivalently, $T$ is not arational if there is a proper $(G,\calf)$-free factor $H$ containing
non-peripheral elements with arbitrarily small (maybe zero) translation length.

 We recall that the \emph{observers' topology} on an $\mathbb{R}$-tree is the topology for which connected components of the complement of finite sets form a basis. We say that two arational trees $T,T'$ are \emph{$\mathcal{AT}$-equivalent} if they are equivariantly homeomorphic when equipped with the observers' topology, 
i.e.\ if there exist  equivariant alignment-preserving bijections between them; we write $T\simAT T'$ in this case. The following lemma says that this equivalence relation is the restriction of the equivalence relation on $\calz$-averse trees.

\begin{lemma}(\cite[Corollary~13.4]{GH15-1}) \label{atza}
Let $T,T'$ be two arational trees. Then 
$T\simAT T'$ if and only if $T\simZA T'$.
\\ More generally, let $T$ be an arational tree, and $T'\in\overline{\calo}$.
If $T$ and $T'$ are compatible then $T'$ is arational and $T\simAT T'$.
\end{lemma}

We work with the model $\FF_2$ (Definition \ref{dfn_ff2}) whose vertices are $\Z$-splittings, and two vertices are joined by an edge if
they are compatible or have a common nonperipheral simple elliptic element.
Let $\psi_{\FF}:\mathcal{O}\to \FF$ be the map that consists in forgetting the metric (for the model $\FF_2$; actually, the statement below is unaffected if $\psi_{\FF}$ is replaced by a map at bounded distance).
In fact, forgetting the metric also defines a map assigning a point in $\FF_2$ to any simplicial tree in $\ol\calo$;
we also denote this map by $\psi_\FF$.

\begin{theo}\label{boundary-ff}
There is a unique $\text{Out}(G,\mathcal{F})$-equivariant homeomorphism $$\partial\psi_{\FF}:\mathcal{AT}/\!\raisebox{-.2em}{$\simAT$}\to\partial_{\infty}\FF$$ such that for all $T\in\mathcal{AT}$ and all sequences $(T_n)_{n\in\mathbb{N}}\in\mathcal{O}^{\mathbb{N}}$ converging to $T$, the sequence $(\psi_{\FF}(T_n))_{n\in\mathbb{N}}$ converges to $\partial\psi_{\FF}(T)$. 
\\ Moreover, if $(T_n)_{n\in\mathbb{N}}\in\mathcal{O}^{\mathbb{N}}$ converges to a non-arational tree $T$, then $\psi_{\FF}(T_n)$ has no accumulation point in $\partial_{\infty}\FF$.
\end{theo}

Here is a sketch of the proof of Theorem \ref{boundary-ff} given below. The natural map 
$j:\ZS\ra \FF$ induced by the inclusion (using the model $\FF_2$) coarsely preserves alignment by Proposition~\ref{fz-ff}.
So by Theorem \ref{dt}, the Gromov boundary $\partial_\infty \FF$ is identified with a subset $\partial_\FF \ZF\subseteq \partial_\infty \ZS$
by a natural  homeomorphism  
$\partial j:\partial_\FF \ZF \ra \partial_\infty \FF$.
Now by Theorem \ref{boundary-fz}, $\partial_\infty \ZS$ is identified with a quotient of the set $\ZA$ of $\calz$-averse trees in $\overline \calo$,
and the corresponding map $\partial\psi_{\ZS}:\ZA\ra \partial_{\infty} \ZS$ is a continuous extension of $\psi_{\ZS}$ (in the precise sense of Theorem~\ref{boundary-fz}). 
Our main task is therefore to show that the set of (equivalence classes of) $\Z$-averse trees corresponding to
the subset $\partial_{\FF}\ZS$ of $\partial_\infty \ZS$
is the set of (equivalence classes of) arational trees: Proposition~\ref{non-arational-visible} will show that non-arational trees are not in $\partial_{\FF}\ZS$, and Proposition~\ref{arational-at-infinity} will show that arational trees belong to $\partial_{\FF} \ZS$. 

\paragraph*{Non-arational trees are not at infinity of $\FF$.} We recall that a tree $T\in\overline{\calo}$ is \emph{mixing} if given any two segments $I,J\subseteq T$, there exist finitely many elements $g_1,\dots,g_k\in G$ such that $J\subseteq g_1I\cup\dots\cup g_kI$.

\begin{lemma}\label{lem_ZAminusAT}
Let $T\in\overline{\calo}$ be a mixing $\calz$-averse tree.  
If $A\subset G$ is a proper $(G,\calf)$-free factor acting with dense orbits on its minimal subtree in $T$, then $A$ is elliptic.
\\ In particular, any $\Z$-averse tree which is not arational has a collapse in which some proper free factor is elliptic.
\end{lemma}

\begin{proof} Let $T_A$ be the $A$-minimal subtree on which $A$ acts with dense orbits by assumption.
Assume by contradiction that $T_A$ is not reduced to a point. By \cite[Corollary~11.8]{GH15-1}, the intersection $gT_A\cap T_A$ is reduced to a point for all $g\notin A$. Since $T$ is mixing, the $G$-translates of $T_A$ thus form a transverse covering of $T$ (i.e.\ every segment of $T$ can be covered by finitely many subtrees from this family). Since $T$ is $\calz$-averse, \cite[Proposition~5.20]{Hor14-2} implies that the stabilizer of $T_A$ (i.e.\ $A$) is not elliptic in any $\calz$-splitting of $(G,\calf)$, contradicting the fact that $A$ is a $(G,\calf)$-free factor.

We now prove the second part of the lemma. Let $T\in\overline{\calo}$ be a $\calz$-averse tree which is not arational. By \cite[Proposition~6.3]{Hor14-2}, the tree $T$ collapses onto a mixing $\calz$-averse tree $\overline{T}$. 
Since $T$ is not arational, there is a proper $(G,\calf)$-free factor $A'$ 
containing non-peripheral elements with arbitrarily small translation length in $T$, and therefore in $\ol T$. In particular $\overline{T}$ is not arational.
Let  $\overline{T}_{A'}\subseteq \overline{T}$ be the minimal $A'$-invariant subtree and $S$ be the Levitt decomposition of $\overline{T}_{A'}$. 
Since $\overline{T}$ (whence $\overline{T}_{A'}$) has trivial arc stabilizers, every vertex group $A$ of $S$ is an $(A',\calf_{|A'})$-free factor acting with dense orbits (possibly trivially) on its minimal invariant subtree $\overline{T}_{A}\subset \overline{T}$. 
Since $A'$ contains non-peripheral elements with arbitrarily small translation length,
one of these vertex groups $A$ is a proper $(G,\calf)$-free factor.  Since $\overline{T}$ is mixing and $\calz$-averse,
the first part of the lemma implies that $A$ is elliptic in $\overline{T}$.
\end{proof}

\begin{prop}\label{non-arational-visible}
Let $T\in\ZA\setminus\mathcal{AT}$.
Then there exist $T'\simZA T$ and a sequence of simplicial metric trees $S_n\in \calo$
converging to $T'$, such that $\psi_{\FF}(S_n)$ is bounded in $\FF$.
\end{prop}

\begin{proof}
Let $\ol T$ be a collapse of $T$ in which some proper $(G,\calf)$-free factor $A$ is elliptic (this exists by Lemma~\ref{lem_ZAminusAT}). Since $T$ (whence $\ol T$) has trivial arcs stabilizers, the tree $\ol T$ can be seen as a very small $(G,\calf\vee \{A\})$-tree (where $\calf\vee\{A\}$ is the free factor system induced by $\calf$ and $A$: $g$ is peripheral relative to $\calf\vee\{A\}$ if and only if
it is peripheral relative to $\calf$ or conjugate in $A$).
By \cite{Hor14-1}, the tree $\ol T$ can therefore be approximated by a sequence of simplicial trees $S'_n$ in
the outer space $\calo(G,\calf\vee\{A\})$. Since $A$ is elliptic in each tree $S'_n$, the image in $\FF$ of the sequence $(S'_n)_{n\in\bbN}$ is bounded.
Now consider $S_n\in\calo$ a  Grushko $(G,\calf)$-tree which is a blowup of $S'_n$. 
The images of $S'_n$ and $S_n$ are at distance at most $1$ in $\FF$, so $\psi_{FF}(S_n)$ is bounded.
Up to taking a subsequence, one may assume that $S_n$ converges projectively to some $T'\in\ol\calo$, and since compatibility is a closed relation
by \cite[Corollary~A.12]{GL16}, $T'$ is compatible with $\ol T$, so $T'\simZA T$ by \cite[Theorem~5.2]{Hor14-2}.
\end{proof}

\paragraph*{Arational trees are at infinity of $\FF$.}

\begin{prop}\label{arational-at-infinity}
Let $T\in\mathcal{AT}$, and let $(S_n)_{n\in\mathbb{N}}\in\overline{\mathcal{O}}^{\mathbb{N}}$ be a sequence of simplicial metric trees that converges to $T$. Then the sequence $(\psi_{\FF}(S_n))_{n\in\mathbb{N}}$ is unbounded in $\FF$.
\end{prop}

The proof of Proposition \ref{arational-at-infinity} is a variation over the argument in \cite{Kob88} for proving unboundedness of the curve complex of a compact, connected, oriented surface. It relies on the following statement, which is one of the main results of our previous paper \cite{GH15-1}.

\begin{theo}[{\cite[Corollary~13.3]{GH15-1}}]\label{gh}
Let $(T_n)_{n\in\mathbb{N}},(T'_n)_{n\in\mathbb{N}}\in\ol{\calo}^{\mathbb{N}}$, such that for all $n\in\mathbb{N}$, some nonperipheral simple element $g_n\in G$ is elliptic in both $T_n$ and $T'_n$. 
\\ If $(T_n)_{n\in\mathbb{N}},(T'_n)_{n\in\mathbb{N}}$ converge respectively to $T,T'$ in $\ol{\calo}$,
and if $T\in\mathcal{AT}$, then $T'\in\mathcal{AT}$, and $T'\simAT T$.
\end{theo}

\begin{proof}[Proof of Proposition \ref{arational-at-infinity}]
We take $\FF=\FF_2$ as our model of the free factor graph.
Let $T\in\mathcal{AT}$, and let $(S_n)_{n\in\mathbb{N}}\in\overline{\mathcal{O}}^{\mathbb{N}}$ be a sequence of simplicial (metric) trees that converges to $T$ (for the topology of $\overline{\calo}$). 
We let $S_n^0=\psi_{\FF}(S_n)$ be the point in $\FF$ corresponding to $S_n$.

Assume towards a contradiction that $(S_n^0)_{n\in\mathbb{N}}$ is bounded in $\FF$. 
Fix a base point $U\in \FF$. Up to passing to a subsequence, we can assume that there exists an integer $M\in\mathbb{N}$ so that for all $n\in\mathbb{N}$, one has $d_{\FF}(U,S_n^0)=M$. For each $n\in\mathbb{N}$, we choose a geodesic segment $S_n^0,S_n^1,\dots,S_n^M=U$ from $S_n^0$ to $U$ in $\FF$. Up to passing to subsequences again, we can assume that for every $i\in\{0,\dots,M\}$, the sequence $(S_n^i)_{n\in\mathbb{N}}$ (where the trees are equipped with arbitrary metrics) converges projectively to a tree $T^i\in\mathbb{P}\overline{\mathcal{O}}$. Up to passing to a further subsequence, we can assume that for all $i\in\{0,\dots,M-1\}$, either 
\begin{enumerate}\renewcommand{\theenumi}{(\roman{enumi})}\renewcommand{\labelenumi}{\theenumi}
\item \label{it_simple} for all $n\in\mathbb{N}$, there exists a nonperipheral simple element $g_n^i\in G$ that is elliptic in both $S_n^i$ and $S_n^{i+1}$, or
\item \label{it_compat} for all $n\in\mathbb{N}$, the trees $S_n^i$ and $S_n^{i+1}$ are compatible.
\end{enumerate}

We prove by induction that for all $i\leq M$, the tree $T^i$ is arational (and in fact that $T^i\simAT T$). 
First, $T^0=T$ is arational by assumption.
So assume $T^i$ is arational.
In case \ref{it_simple}, Theorem~\ref{gh} ensures that $T^{i+1}\in\AT$. In case \ref{it_compat}, by taking the limit, we get that the two trees $T^i$ and $T^{i+1}$ are compatible \cite[Corollary~A.12]{GL16}; in particular $T^{i+1}\in\AT$ by Lemma~\ref{atza}. 
We thus deduce that $T^M$ is arational. But on the other hand, $T^M$ is the limit of the constant sequence $S_n^M=U$, so $T_M=U$ is not arational.
This contradiction concludes the proof.
\end{proof}

\paragraph*{End of the proof.}

\begin{proof}[Proof of Theorem \ref{boundary-ff}]
Recall that we denote by $j:\ZS\ra \FF$ the map induced by the inclusion.
It fits in a commutative diagram
$$
 \xymatrix{
   \calo\ar[d]_{\psi_{\ZS}}\ar[dr]^{\psi_{\FF}}& \\
   \ZS\ar[r]^j&\FF.
 }
 $$

Recall that $\partial_{\FF}\ZS$  is the  subset of $\partial_\infty \ZS$ defined as follows:
$\xi\in \partial_\FF \ZS$ if and only if every sequence $(S_n)_{n\in\mathbb{N}}\in\ZS^{\mathbb{N}}$ converging to $\xi$ projects to an unbounded sequence $j(S_n)$ in $\FF$ (see Lemma \ref{lem_eqv_bord}).
By Theorem~\ref{dt} and Proposition~\ref{fz-ff}, $j$ extends to a homeomorphism $\partial j$ from $\partial_{\FF}\ZS\subseteq\partial_\infty \ZS$ to $\partial_\infty \FF$.

By Theorem~\ref{boundary-fz}, there is a homeomorphism  $\partial \psi_{\ZS}$ that allows to identify $\partial_\infty \ZS $ with the set of $\simZA$-classes of $\Z$-averse trees. We claim that $\partial_{\FF}\ZS$ is the image of the set of arational trees under this identification.
Indeed, if $\xi\in\partial_\infty\ZS$ is the image of a tree $T\in\ZA$ which is $\Z$-averse but not arational, 
Proposition \ref{non-arational-visible} provides a sequence of trees $S_n$ in $\calo$ converging to some $T'\simZA T$
whose image in $\FF$ is bounded. 
By the continuity properties of $\partial \psi_{\ZS}$ (Theorem~\ref{boundary-fz}), $\psi_{\ZS}(S_n)$ converges to $\xi$ in $\ZS\cup\partial_\infty\ZS$, 
so $\xi\notin\partial_{\FF}\ZS$ and $\partial_\FF\ZS\subset \partial\psi_\ZS(\AT)$. 
To prove the converse inclusion, consider $\xi=\partial\psi_\ZS(T)$ for some arational tree $T$.
Consider $S_n\in \ZS$ any sequence converging to $\xi$, 
and let $S'_n\in\ZS$ be a metric free splitting compatible with $S_n$ (this exists by Lemma \ref{lem_tirer}). 
It suffices to prove that $j(S'_n)$ hence that $j(S_n)$ is unbounded.
Up to extracting a subsequence, we may assume $(S'_n)_{n\in\bbN}$ converges projectively to some $T'\in\ol\calo$.
By Theorem~\ref{boundary-zf}, $T'$ represents $\xi$ in $\partial_\infty \ZS$, so $T\simZA T'$.
By definition of $\simZA$, this means that $T$ and $T'$ are compatible with a tree  $T''\in\ol\calo$ so $T''$ and $T'$ are arational
by Lemma~\ref{atza}.
By Proposition~\ref{arational-at-infinity}, $j(S'_n)$ is unbounded.
This concludes the proof that $\partial_\FF\ZS$ coincides with the image of the set of arational trees.
By Lemma~\ref{atza}, the equivalence $\simZA$ on $\ZA$ coincides with $\simAT$ in restriction to $\AT$, so 
$$\partial_\infty\FF\simeq \partial_{\FF}\ZS \simeq \AT/\!\raisebox{-.2em}{$\simAT$}.$$ 

The map $\partial\psi_{\FF}$ is obtained by composing the restriction of $\partial\psi_{\ZS}$
with $\partial j$ as in the following diagram:
$$\xymatrix{
  \AT/\!\raisebox{-.2em}{$\simAT$} \ar@{->}[r]^{\simeq} \ar@{-->}[d]^{\partial\psi_{\FF}} 
  &   \AT/\!\raisebox{-.2em}{$\simZA$} \ar@{^(->}[r] \ar@{->}[d]^{(\partial\psi_{\ZS})_{|\AT/_{\simZA}}} 
  &\ZA/\!\raisebox{-.2em}{$\simZA$} \ar[d]^{\partial\psi_{\ZS}} 
\\
\partial_\infty \FF  &\partial_{\FF}\ZS\ar[l]^{\partial j}_{\simeq}\ar@{^{(}->}[r] &\partial_\infty \ZS.
}$$

Let now $(T_n)_{n\in\mathbb{N}}\in\mathcal{O}^\mathbb{N}$ be a sequence that converges to $T\in\mathcal{AT}$.
By Theorem~\ref{boundary-fz}, the sequence $(\psi_{\ZS}(T_n))_{n\in\mathbb{N}}$ converges to the point $\xi\in\partial_\infty \ZS$ corresponding to the equivalence class of $T$. Theorem \ref{dt} then implies that $\psi_{\FF}(T_n)$ converges to $\partial j(T)$, which is equal to $\partial\psi_{\FF}(T)$ by construction.

We finally prove the last assertion of Theorem~\ref{boundary-ff}. Let $(T_n)_{n\in\mathbb{N}}$ be a sequence that converges to a tree $T\in\ol\calo\setminus\AT$. If $T\in\ZA$, then Theorem~\ref{dt} implies that $(\psi_{\FF}(T_n))_{n\in\mathbb{N}}$ has no accumulation point in $\partial_\infty \FF$. If $T\notin\ZA$, then the sequence $(\psi_{\ZS}(T_n))$ has a subsequence along which all Gromov products $(\psi_{\ZS}(T_n)|\psi_{\ZS}(T_m))$ with $n\neq m$ are bounded, and it follows by applying the coarsely alignment preserving map $j$ that the Gromov products $(\psi_{\FF}(T_n)|\psi_{\FF}(T_m))$ are also bounded. This shows that the sequence $(\psi_{\FF}(T_n))$ does not converge in $\partial_\infty \FF$. Since this holds for any subsequence, $(\psi_{\FF}(T_n))$ does not accumulate in $\partial_\infty \FF$.
\end{proof}

\subsection{The Gromov boundary of the $\calz$-factor graph}\label{sec-8}

Again, there is a natural $\Out(G,\calf)$-equivariant map $\psi_{\ZF}:\mathcal{O}\to \ZF$, which just consists in forgetting the metric.
We denote by $\FAT$ the subspace of $\mathcal{AT}$ consisting of relatively free actions. In fact, every arational tree which is not relatively free, is an arational surface tree in the following sense \cite{Rey12,Hor14-3}.

\begin{de}[\emph{\textbf{Arational surface tree}}]\label{dfn_arat-surf}
A tree $T\in\ol\calo$ is an \emph{arational surface tree} if it splits as a graph of actions over a geometric decomposition with a single unused boundary curve (see Figure   \ref{fig-arat-surf}), so that the action corresponding to the vertex associated to the orbifold $\Sigma$ is dual to an arational measured lamination on $\Sigma$. 
\end{de}

It was proved in \cite[Section 4.1]{Hor14-3} that arational surface trees are indeed arational.

\begin{theo}\label{boundary-zf}
There is a unique $\text{Out}(G,\mathcal{F})$-equivariant homeomorphism $$\partial\psi_{\ZF}:\FAT/\!\raisebox{-.2em}{$\simAT$}\to\partial_{\infty}\ZF$$ such that for all $T\in\FAT$ and all sequences $(T_n)_{n\in\mathbb{N}}\in\mathcal{O}^{\mathbb{N}}$ converging to $T$, the sequence $(\psi_{\ZF}(T_n))_{n\in\mathbb{N}}$ converges to $\partial\psi_{\ZF}(T)$. 
\\ Moreover, if $(T_n)_{n\in\mathbb{N}}\in\mathcal{O}^{\mathbb{N}}$ converges to a tree $T\in\overline{\calo}\setminus\FAT$, then $\psi_{\ZF}(T_n)$ has no accumulation point in $\partial_{\infty}\ZF$.
\end{theo}

\begin{rk}
  If $(G,\calf)$ is sporadic or isomorphic to $(F_2,\emptyset)$, then $\ZF$ is bounded and $\FAT$ is empty so the result is valid but empty.
\end{rk}

\begin{proof}
  The proof is the same as the proof of Theorem \ref{boundary-ff}, by
  applying Dowdall--Taylor's criterion (Theorem \ref{dt}) to the
  coarsely alignment-preserving map $\psi:\FF\to \ZF$, and replacing
  Propositions \ref{non-arational-visible} and
  \ref{arational-at-infinity} with Propositions \ref{non-fat} and
  \ref{fat} below.
\end{proof}

\begin{prop}\label{non-fat}
Let $(G,\calf)$ be nonsporadic and not isomorphic to $(F_2,\es)$.
Let $T\in\mathcal{AT}$ be an arational surface tree. 
\\ Then there exists a sequence of simplicial metric trees 
$(S_n)_{n\in\mathbb{N}}\in\mathcal{O}^{\mathbb{N}}$ that converges to a tree $T'\simAT T$, such that the sequence $(\psi_{\ZF}(S_n))_{n\in\mathbb{N}}$ is bounded in $\ZF$.
\end{prop}

\begin{proof}
Let $\Sigma$ be the orbifold associated to $T$, and let $L$ be the corresponding measured lamination as in Definition \ref{dfn_arat-surf}. We denote by $\grp{c_0}$
the fundamental group of the unused boundary curve of $\Sigma$. We denote by $\mathcal{PML}(\Sigma)$ the space of projective measured laminations on $\Sigma$ (for some fixed hyperbolic metric on $\Sigma$). We can find a sequence of weighted essential simple closed geodesic multicurves $(\gamma_n)_{n\in\mathbb{N}}\in\mathcal{PML}(\Sigma)^{\mathbb{N}}$ that converges to $L$. These multicurves $\gamma_n$ are dual to $(G,\mathcal{F})$-trees $T_{\gamma_n}$ with cyclic edge stabilizers, which converge to $T$ in the Gromov--Hausdorff topology.

We claim that there exists an essential two-sided geodesic $\gamma'_n$ disjoint or equal to  one of the geodesics  in $\gamma_n$.
Indeed, we may assume without loss of generality that $\gamma_n$ consists of a single
one-sided geodesic. If one cannot find such a $\gamma'_n$, 
then $\Sigma\setminus \gamma_n$ is either 
 a sphere with at most three punctures or conical points, or a projective plane with at most two punctures or conical points \cite[Section 5.1.4]{GL16}.
Therefore, $\Sigma$ is either a projective plane with at most two punctures or conical points, or a Klein bottle with at most
one puncture or conical point.  In each case, one easily checks that $(G,\calf)$ is sporadic or isomorphic to $(F_2,\es)$, which proves the claim.

Now let $T_{\gamma'_n}$ be the tree dual to such a two-sided geodesic.
This is a very small tree and since $c_0$ is elliptic in all trees $T_{\gamma'_n}$,
the image of this sequence in $\ZF$ is bounded.
Let $S'_n$ be a free splitting compatible with $T_{\gamma'_n}$ (dual to an arc on $\Sigma$ with endpoints
on the unused boundary curve of $\Sigma$, or using Lemma \ref{lem_tirer}),
and $S_n\in\calo$ a Grushko splitting refining $S'_n$.
Since compatible trees are at distance one in $\ZF$, the image of $S_n$ in $\ZF$ is still bounded.
Up to passing to subsequences, we may assume that $T_{\gamma'_n}$, $S'_n$ and $S_n$ respectively 
converge projectively to trees $T_1,T_2,T_3$ in $\ol\calo$. 
Denoting $T_0=T$ to uniformize notation one gets
that $T_i$ is compatible with $T_{i+1}$ for $i\in\{0,1,2\}$
because compatibility is a closed relation \cite[Corollary A.12]{GL16}.
By Lemma \ref{atza}, $T_i\sim_\AT T_{i+1}$, which proves the proposition with $T'=T_3$.
\end{proof}

\begin{prop}\label{fat}
Let $T\in\FAT$, and let $(S_n)_{n\in\mathbb{N}}\in\mathcal{O}^{\mathbb{N}}$ be a sequence that converges to $T$. Then the sequence $(\psi_{\ZF}(S_n))_{n\in\mathbb{N}}$ is unbounded in $\ZF$.
\end{prop}

\begin{proof}
Recall that the graphs $\FF$ (with the model $\FF_2$) and $\ZF$ have the same vertices, and all edges of $\FF$ are edges of $\ZF$.
Additional edges in $\ZF$ occur when  there is a non-simple element $g$ that is elliptic in two trees.
The proof of Proposition~\ref{fat} is exactly the same as the proof of Proposition~\ref{arational-at-infinity}, by using the following variation of Theorem~\ref{gh}, in which the elements $g_n$ are no longer required to be simple.
\end{proof}

\begin{theo}(\cite[Corollary~13.2]{GH15-1})\label{gh-2}
Let $(T_n)_{n\in\mathbb{N}},(T'_n)_{n\in\mathbb{N}}\in\ol{\calo}^{\mathbb{N}}$, such that for all $n\in\mathbb{N}$, some nonperipheral element $g_n\in G$ is elliptic in both $T_n$ and $T'_n$. 
\\ If $(T_n)_{n\in\mathbb{N}},(T'_n)_{n\in\mathbb{N}}$ converge respectively to $T,T'$ in $\ol{\calo}$,
and if $T\in\FAT$, then $T'\in\FAT$, and $T'\sim_{\mathcal{AT}} T$.
\end{theo}

\section{Loxodromic isometries of $\FF$ and $\ZF$}\label{sec-loxo}

We now determine which elements of $\Out(G,\calf)$ act loxodromically on either $\FF$ or $\ZF$. In the case where $G=F_N$ and $\calf=\emptyset$, this is due to Bestvina--Feighn \cite{BF14} and Mann \cite{Man14}. In the case where $G=F_N$ and $\calf$ is arbitrary, loxodromic isometries of the free factor graph were determined by Gupta \cite{Gup2} by a different method.

We recall that an outer automorphism $\Phi\in\Out(G,\calf)$ is \emph{fully irreducible} (with respect to $(G,\calf)$) if none of its powers fixes the conjugacy class of a proper $(G,\calf)$-free factor.

\begin{theo}[see also \cite{Gup2}]\label{loxo}
Let $(G,\calf)$ be non-sporadic.
An automorphism $\Phi\in\text{Out}(G,\calf)$ acts loxodromically on $\FF$ if and only if $\Phi$ is fully irreducible (otherwise $\Phi$ acts elliptically on $\FF$).
\end{theo}

\begin{proof}
If $\Phi$ is not fully irreducible, then some power of $\Phi$ preserves the conjugacy class of a proper $(G,\calf)$-free factor, so $\Phi$ acts elliptically on $\FF$.

If $\Phi$ is fully irreducible, \cite[Theorem~8.23]{FM} ensures that $\Phi$ has an invariant folding line $L:\bbR\ra \calo$, and this line projects to a $\Phi$-invariant (unparametrized) quasi-geodesic in $\FF$.
It suffices to prove that the projection of $L$ in $\FF$ is unbounded.
There exists an element $g\in G$ whose axis in $L(0)$ isometrically embeds in $L(t)$ for all $t\geq 0$.
Since the length function of $L(t)$ is non-increasing, 
this implies that $L(t)$ has a limit $T\in\ol\calo$ as $t\ra\infty$.
It follows that $T$ is $\Phi$-invariant. Since $\Phi$ is fully irreducible and $(G,\calf)$ is nonsporadic, this implies
that $T$ is arational because if $T$ is not arational, then one can associate to $T$
a canonical hence $\Phi$-invariant nonempty finite family of proper $(G,\calf)$-free factors
(see \cite{Rey12,Hor14-3}).
By Theorem~\ref{boundary-ff}, this implies that the projection of $L$ in $\FF$ is unbounded.
\end{proof}

 An automorphism $\Phi\in\text{Out}(G,\mathcal{F})$ is \emph{atoroidal} if no power $\Phi^k$ (with $k\neq 0$) fixes a nonperipheral conjugacy class.

\begin{theo}
Let $(G,\calf)$ be non-sporadic.
An automorphism $\Phi\in\text{Out}(G,\calf)$ acts loxodromically on $\ZF$ if and only if $\Phi$ is fully irreducible and atoroidal (otherwise $\Phi$ acts elliptically on $\ZF$).
\end{theo}

\begin{proof}
The proof is the same as the proof of Theorem~\ref{loxo}, by noticing that the limiting tree $T$ cannot be arational surface, as otherwise $\Phi$ would fix the unique conjugacy class of nonperipheral cyclic point stabilizers of $T$. One then uses Theorem~\ref{boundary-zf} in place of Theorem~\ref{boundary-ff}.
\end{proof}

\section{Subgroups with bounded orbits in $\FF$ or $\ZF$}\label{sec-bdd}

The goal of the present section is to prove the following proposition, following arguments from \cite{Hor14-5,Hor14-3}, see also \cite{KM96}.

\begin{prop}\label{bounded-finite-ff}
Assume that $(G,\calf)$ is nonsporadic, and let $H\subseteq \text{Out}(G,\calf)$ be a subgroup. 
\\ If $H$ has bounded orbits in $\FF$, then $H$ virtually fixes the conjugacy class of a proper $(G,\calf)$-free factor.
\end{prop}

Using the model $\FF_3$ for the free factor graph, this can be restated by saying that every subgroup with bounded orbits in $\FF_3$ has a finite orbit (which is not immediate because $\FF_3$ is not locally finite). 

\begin{proof}
Let $H\subseteq\Out(G,\calf)$ be a subgroup. We will show that either $H$ virtually fixes the conjugacy class of a proper $(G,\calf)$-free factor, or else $H$-orbits in $\FF$ are unbounded. 

Let $\mu$ be a probability measure on $H$ which gives positive measure to every element of $H$. Since $\mathbb{P}\overline{\calo}$ is compact and metrizable, the space of all probability measures on $\mathbb{P}\overline{\calo}$ is weakly compact. Let $\nu$ be a probability measure on $\mathbb{P}\overline{\calo}$ obtained as a weak limit of the Cesàro averages of the measures $\mu^{\ast k}\ast\delta_{S_0}$ as $k$ goes to $+\infty$, where $\delta_{S_0}$ is the Dirac mass at a point $S_0\in\mathbb{P}\calo$. Then the measure $\nu$ is $\mu$-stationary, i.e.\ for every measurable subset $A\subseteq\mathbb{P}\overline{\calo}$, one has $$\nu(A)=\sum_{h\in H}\mu(h)\nu(h^{-1}A).$$ We will discuss two cases, depending on whether $\nu$ gives full measure to the projectivized set $\mathbb{P}\AT$ of  arational trees. 

We first assume that $\nu$ gives positive measure to $\mathbb{P}\AT$. Since $\nu(\overline{H\cdot S_0})=1$ by construction,  
it follows that $\overline{H\cdot S_0}\cap\mathbb{P}\AT\neq\emptyset$. Theorem~\ref{boundary-ff} thus implies that the $H$-orbit of $\psi_{\FF}(S_0)$ accumulates to a point in $\partial_\infty \FF$, so $H$-orbits in $\FF$ are unbounded.

We now assume that $\nu$ gives positive measure to $\mathbb{P}\overline{\calo}\setminus\mathbb{P}\AT$, and we will show that $H$ virtually fixes the conjugacy class of a proper $(G,\calf)$-free factor. Let $\cald$ be the set of all finite nonempty collections of proper $(G,\calf)$-free factors. By \cite{Rey12,Hor14-3}, there exists a measurable $\Out(G,\calf)$-equivariant map that assigns to every non-arational tree $T$ a nonempty finite set $\mathrm{Red}(T)$ of proper $(G,\calf)$-free factors. Therefore, by pushing forward the measure $\nu$, we get a finite nonzero $\mu$-stationary measure $\overline{\nu}$ on the countable set $\cald$. The collection of all elements of $\cald$ with maximal $\overline{\nu}$-measure is then finite (as $\overline{\nu}$ is a finite measure) and $H$-invariant (because $\overline{\nu}$ is stationary and $\mu$ gives positive measure to every element of $H$). This yields an $H$-periodic conjugacy class of proper $(G,\calf)$-free factor, as desired. 
\end{proof}

Using the same method, we also prove an analogous statement for subgroups of $\Out(G,\calf)$ with bounded orbits in the $\calz$-factor graph. We refer to Definition~\ref{def-quadratic} for the definition of quadratic elements.

\begin{prop}\label{bounded-finite-zf}
Assume that $(G,\calf)$ is nonsporadic, and let $H\subseteq\Out(G,\calf)$ be a subgroup.
\\ If $H$ has bounded orbits in $\ZF$, then $H$ virtually preserves the conjugacy class of a proper $(G,\calf)$-free factor, or a quadratic conjugacy class.
\end{prop}

\begin{proof}
We follow the same strategy as in the proof of Proposition~\ref{bounded-finite-ff}. We will show that either $H$ virtually preserves the conjugacy class of a proper $(G,\calf)$-free factor, or $H$ virtually fixes a quadratic conjugacy class, or $H$-orbits in $\ZF$ are unbounded.

Let $\mu$ be a probability measure on $H$ such that every element has positive $\mu$-measure, let $S_0\in\mathbb{P}\calo$, and let $\nu$ be a $\mu$-stationary measure on $\mathbb{P}\overline{\calo}$ supported on the closure of the $H$-orbit of $S_0$. 

We first assume that $\nu$ gives positive measure to $\mathbb{P}\FAT$. By projecting to $\ZF$, and using the description of $\partial_\infty\ZF$ in terms of free arational trees given by Theorem~\ref{boundary-zf}, this implies that $H$-orbits in $\ZF$ are unbounded.

We now assume that $\nu$ gives positive measure to $\mathbb{P}\overline{\calo}\setminus\mathbb{P}\AT$. The same argument as in the above proof then shows that $H$ virtually fixes the conjugacy class of a proper $(G,\calf)$-free factor.

We finally assume that $\nu$ gives positive measure to $\mathbb{P}\mathcal{AT}\setminus\mathbb{P}\FAT$. Let $\cald$ be the countable set of all finite collections of conjugacy classes of cyclic subgroups generated by quadratic elements. We define an $\Out(G,\calf)$-equivariant measurable map $\theta:\mathbb{P}\mathcal{AT}\setminus\mathbb{P}\FAT\ra\cald$ by letting $\theta(T)$ be the unique conjugacy class of nonperipheral point stabilizer in $T$ (recall that $\mathbb{P}\mathcal{AT}\setminus\mathbb{P}\FAT$ is the space of projective classes of arational surface trees). By forward-pushing the measure $\nu$, we get a finite nonzero $\mu$-stationary measure $\overline{\nu}$ on $\cald$. The set of elements of $\cald$ of maximal $\overline{\nu}$-measure is then finite and $H$-invariant, which shows that $H$ virtually fixes a quadratic conjugacy class.
\end{proof}

\section{Stabilizers of arational trees}\label{sec-5}

Throughout the section, we assume that $(G,\calf)$ is nonsporadic. The goal of this section is to establish the dichotomy in Theorem~1 from the introduction in the particular case of subgroups of $\text{Out}(G,\mathcal{F})$ which stabilize the homothety class of an arational tree in the closure of the relative outer space.  

Given $T\in\ol\calo$, we distinguish two stabilizers.
The \emph{homothetic stabilizer} $\Stab([T])$ 
is the stabilizer of $[T]$ for the action of $\Out(G,\calf)$ on the projectivized outer space $\bbP\ol\calo$.
Equivalently, $\Phi\in\Out(G,\calf)$ lies in $\Stab([T])$ if there exists a lift $\Tilde\Phi\in\Aut(G)$ representing $\Phi$,
and a homothety $I_{\tilde \Phi}:T\ra T$ which is $\tilde \Phi$-equivariant (i.e.\ $I_{\tilde\Phi}(gx)=\tilde\Phi(g)I_{\tilde\Phi}(x)$ for all $g\in G$ and all $x\in T$).
The homothety $I_{\tilde\Phi}$ is unique, and its existence does not depend on the choice of a lift $\tilde\Phi$ of $\Phi$ (indeed, if $\ad_g$ denotes the inner automorphism $h\mapsto ghg\m$, 
then $I_{\ad_g\circ\Tilde\Phi}$ is the map $x\mapsto g.I_{\Tilde\Phi}(x)$).
In particular, the scaling factor of $I_{\Tilde\Phi}$ does not depend on the choice of the lift of $\Phi$, and we denote it by $\lambda_T(\Phi)$.
The map $\Phi\mapsto\lambda_T(\Phi)$ is a morphism $\Stab([T])\ra \bbR_+^*$.
The kernel of this morphism is called the \emph{isometric stabilizer} of $T$
which we denote by $\Stabis(T)$.
It is the stabilizer of $T$ for the action of $\Out(G,\calf)$ on unprojectivized outer space $\ol\calo$. 

The goal of this section is to prove the following statement.

\begin{theo}\label{thm_stabT}
Let $T\in\AT$ be an arational $(G,\calf)$-tree, and let $H\subseteq\Stab([T])$. Then
\begin{enumerate}
\item either $H\subseteq\Stab^{\mathrm{is}}(T)$, or else there exists a fully irreducible element $\Phi\in H$ such that $H=H^1\rtimes\langle \Phi\rangle$, with $H^1\subseteq\Stab^{\mathrm{is}}(T)$;
\item if $H\subseteq\Stab^{\mathrm{is}}(T)$, then $H$ contains no fully irreducible and no atoroidal element; if additionally $H$ has finite fix type (see Definition~\ref{dfn_fft} below), then $H$ virtually fixes a free splitting of $(G,\calf)$, and there is a nonperipheral subgroup $F\subseteq G$ which is not virtually cyclic and such that $H$ is virtually contained in $\Out(G,F^{(t)})$.
\end{enumerate}
\end{theo}

We prove intermediate lemmas before proving the theorem.

\begin{lemma}[\cite{GL}]\label{lem_cyclic} 
  For any $T\in \ol\calo$, the image of the morphism $\lambda_T$ is a cyclic (maybe trivial) subgroup of $\bbR_+^*$.
  \\ In particular, if $H\subseteq\Stab([T])$, then either $H\subseteq \Stabis(T)$ or
  $H=H^1\rtimes \grp{\Phi}$ for some $\Phi\in H$ with $\lambda_T(\Phi)\neq 1$
  and $H^1=H\cap\Stab^{\mathrm{is}}(T)$.
\end{lemma}

The following proposition applies when $\Stab([T])$ contains a homothety with nontrivial scaling factor.

\begin{prop}\label{stabilizers_homothetie} 
Let $T\in\mathcal{AT}$, and let $\Phi\in\Stab([T])$ with $\lambda_T(\Phi)\neq 1$.
\\ Then $\Phi$ is fully irreducible.
\end{prop}

\begin{rk}\label{rk_scaling}
  Corollary \ref{cor_totalement_reductible} will show that the converse also holds: if $\Phi$  stabilizes $[T]$ and is fully irreducible, then $\lambda_T(\Phi)\neq 1$. 
\end{rk}

\begin{proof}
Assume towards a contradiction that there exists $k\neq 0$ such that $\Phi^k$ preserves the conjugacy class of a proper $(G,\mathcal{F})$-free factor $A$. 
Since $T$ is arational, the action of $A$ on $T$ is simplicial and relatively free.
Let $g\in A$ be hyperbolic in $T$. By applying powers of $\Phi^{\pm k}$ to $g$, we build a sequence of conjugacy classes of elements in $A$ 
with arbitrary small translation length in $T$, a contradiction. 
\end{proof}

Given a group $P$ and a subgroup $H\subseteq\Aut(P)$, we denote by $\Fix_P(H)$ the subgroup of $P$ consisting of all elements $p\in P$ fixed by all automorphisms in $H$. We say that a subgroup of $P$ of the form $\Fix_P(H)$ for some subgroup $H\subseteq\Aut(P)$ is a \emph{fixed group}.

  Let $T\in\ol\calo$ with trivial arc stabilizers, and let $\eta$ be a direction
  at a point $v\in T$. 
  If $\alpha \in \Stabis(T)$, one says that $\alpha$ preserves the orbit of $v$ (resp.\ of $\eta$)
  if for all lifts $\tilde \alpha\in\Aut(G)$ of $\alpha$ (equivalently for some lift),
  $I_{\tilde\alpha}$ preserves the orbit of $v$ (resp.\ of $\eta$).
  If $H\subseteq \Stabis(T)$, let $H_v$ (resp.\ $H_\eta$) be the finite index subgroup
  of $H$ consisting of all $\alpha$ preserving the orbit of $v$ (resp.\ of $\eta$).
  There is a natural map $r_v:H_v\ra \Out(G_v)$ obtained by choosing a lift $\tilde\alpha$ of $\alpha\in H_v$
preserving $G_v$, and looking at the restriction $\tilde\alpha_{|G_v}$ which is well defined up to an inner automorphism of $G_v$. 
  Associated to the direction $\eta$, the restriction $r_{v|H_\eta}:H_\eta\ra \Out(G_v)$ 
has a natural lift $\tilde r_\eta:H_\eta\ra \Aut(G_v)$ defined as follows:
  $\tilde r_\eta(\alpha)=\tilde \alpha_{|G_v}$ where $\tilde \alpha\in\Aut(G)$ is the unique lift of $\alpha$ such
  that $I_{\tilde\alpha}$ fixes the direction $\eta$.

We denote by $\Stab^{\mathrm{is},0}(T)$ the finite index subgroup of $\Stabis(T)$ preserving each orbit of direction at branch points of $T$.

\begin{de}\label{dfn_fft_wrt_T}
  Let $T\in\ol\calo$ with trivial arc stabilizers, and let $H\subseteq\Stabis(T)$ be a subgroup.
  Let $H'=H\cap\Stab^{\mathrm{is},0}(T)$. 
\\  We say that $H$ has \emph{finite fix type with respect to $T$} if for
every direction $\eta$ at a branch point of $T$, 
there exists a finitely generated subgroup $H'_0\subseteq H'$ such that $\Fix_{G_v} (\tilde r_\eta(H'_0))=\Fix_{G_v} (\tilde r_\eta(H'))$.
\end{de}

Equivalently, writing $H'$ as an increasing union of finitely generated groups $H'_k$,
$H$ has finite fix type if for any $\tilde r_\eta$, the descending sequence
of subgroups $\Fix_{G_v} (\tilde r_\eta(H'_k))$ is stationary. 

\begin{de}\label{dfn_fft}
  A subgroup $H\subseteq\Out(G,\calf)$ \emph{has finite fix type} if for every arational $(G,\calf)$-tree $T$ such that $H$ is virtually contained
in $\Stab^{\mathrm{is}}(T)$,
  $H\cap \Stab^{\mathrm{is}}(T)$ has finite fix type with respect to $T$.
\end{de}

Recall that a toral relatively hyperbolic group is a torsion-free group which is hyperbolic relative to a finite collection of finitely generated abelian subgroups. The following lemma gives concrete situations of subgroups of $\Out(G,\calf)$ of finite fix type.

\begin{lemma}\label{lem_FFT_RH_FG}
Let $H\subseteq\Out(G,\calf)$ be a subgroup. If $H$ is finitely generated, or if $G$ is a toral relatively hyperbolic group, then $H$ has finite fix type. 
\end{lemma}

\begin{proof}
If $H$ is finitely generated, the lemma is obvious.

Assume that $G$ is a toral relatively hyperbolic group.
It was proved in \cite{GL_McCool} that there is a bound on the length of any increasing chain of fixed subgroups 
in $G$ (the case where $G$ is a free group was proved earlier by Martino--Ventura \cite{MV}). 
In particular, for any subgroup $\tilde K\subset\Aut(G)$,
there is a finitely generated subgroup $\tilde K_0\subset \tilde K$
such that $\Fix_{G}(\tilde K_0)=\Fix_G(\tilde K)$ hence
$\Fix_{G}(\tilde K_0)\cap G_v=\Fix_{G}(\tilde K)\cap G_v$ for any subgroup $G_v\subseteq G$.
It follows that for any arational $(G,\calf)$-tree $T$, and for any subgroup $H$ virtually contained in $\Stabis(T)$, $H$ has finite fix type with respect to $T$.
\end{proof}

Here is another example. Say that a group $P$ satisfies the \emph{descending chain condition for centralizers}
if for any ascending sequence of subgroups $P_1\subseteq P_2\subseteq \cdots $,
the sequence of centralizers $Z_P(P_1)\supseteq Z_P(P_2)\supseteq \cdots $ is stationary.
This condition holds for any linear group for instance. If it holds for a group $P$, it also holds for any subgroup $P'\subset P$.
We also note that $G$ satisfies the descending chain condition of centralizers if and only every
peripheral subgroup $G_i$ does.

\begin{lemma}\label{lem_FFT_centralisateurs}
Assume that all peripheral subgroups $G_i$ satisfy the descending chain condition for centralizers.

Then any subgroup of $\Out(G,\calf^{(t)})$ has finite fix type.
\end{lemma}

\begin{proof}
Let $H$ be a subgroup of $\Out(G,\calf^{(t)})$.
Let $T$ be an arational $(G,\calf)$-tree such that $H$ is virtually contained
in $\Stab^{\mathrm{is}}(T)$.
Let $H'=H\cap\Stab^{\mathrm{is},0}(T)$.
Let $\eta$ be a direction at a branch point $v\in T$.
Let $\tilde r_\eta:H'\ra \Aut(G_v)$ be the corresponding morphism.
Since $T$ is arational, $G_v$ is either a group in $\calf$,
or a cyclic group. If $G_v$ is cyclic, its fixed subgroups are obvious, and there is nothing to do.
Since $H'\subset\Out(G,\calf^{(t)})$, $\tilde r_\eta(H')$ is contained
in $\Inn(G_v)$. Since fixed subgroups of inner automorphisms are centralizers of elements of $G_v$, the descending chain condition on centralizers
implies that $\Fix_{G_v}(\tilde r_\eta(H'))=\Fix_{G_v}(\tilde r_\eta(H'_0))$ 
for some finitely generated subgroup $H'_0\subset H'$.
\end{proof}

Recall that a \emph{transverse covering} of an $(G,\calf)$-tree $T$  is a $G$-invariant collection $\caly$ of nondegenerate subtrees of $T$ such that any two distinct trees in $\caly$ intersect in at most one point, and every segment in $T$ is covered by finitely many subtrees from the family $\caly$.

\begin{theo}[\cite{GL}]\label{thm_pG}
Let $T$ be an arational $(G,\calf)$-tree, and $H\subseteq\Stab^{\mathrm{is}}(T)$.

If $H$ has finite fix type with respect to $T$, then $H$ has a finite index subgroup $H^0\subseteq H$ which is uniformly piecewise $G$ in the following sense:
there exists a transverse covering $\caly$ of $T$ such that
for every $Y\in\caly$ and every $\tilde\alpha$ in the preimage  $\Tilde H^0$ of $H^0$ in $\Aut(G)$, 
there exists $g\in G$ such that for every $x\in Y$, one has $I_{\tilde \alpha}(x)=gx$.
\end{theo}

\begin{prop}\label{stabilizers_isom} 
Let $T\in\mathcal{AT}$, $H\subseteq\Stabis(T)$, and assume that $H$ has finite fix type.
\\ Then $H$ virtually fixes a $(G,\calf)$-free splitting, and in particular a proper $(G,\mathcal{F})$-free factor. 
\\ Moreover, there is a non-peripheral subgroup $F\subseteq G$ which is not virtually cyclic such that $H$ is virtually contained in $\Out(G,F^{(t)})$.
\end{prop}

\begin{rk}
We do not know whether it is possible to remove the hypothesis on $H$ in the statement.

The last assertion of the proposition can be reformulated as follows: 
there is a finite index subgroup $H^0\subseteq H$ such that
each element $\alpha\in H^0$ has a representative $\tilde\alpha\in\Aut(G)$ whose restriction to $F$ is the identity.
In particular, $H^0$ preserves the conjugacy class of every element of $F$.
\end{rk} 

\begin{proof}
By Theorem \ref{thm_pG}, there is a finite index subgroup $H^0\subseteq H$ which is uniformly piecewise-$G$. 
Let $\caly$ be a transverse covering of $T$
such that for every $Y\in\caly$ and every $\tilde\alpha$ in the preimage  $\Tilde H^0$ of $H^0$ in $\Aut(G)$, 
there exists $g\in G$ such that for every $x\in Y$, one has $I_{\tilde \alpha}(x)=gx$.

As $T$ is arational, it is mixing \cite{Rey12,Hor14-3}, and therefore all the subtrees in $\caly$ are in the same $G$-orbit.
Let $S$ be the skeleton of this transverse covering, as defined in \cite[Definition~4.8]{Gui_limit}: this is the simplicial tree having one vertex $v_Y$ for every subtree $Y\in\caly$, one vertex $v_x$ for every point $x\in T$ that belongs to at least two subtrees in $\caly$, and an edge between $v_x$ and $v_Y$ whenever $x\in Y$. Since $I_{\tilde\alpha}$ preserves the transverse covering, it induces an $\tilde\alpha$-equivariant automorphism $J_{\tilde \alpha}$ of $S$.

We claim that edge stabilizers of $S$ are peripheral: each edge $\eps$ of $S$ corresponds to a pair $(x,Y)$ with $Y\in\caly$ and $x\in Y$, so $G_\eps\subset G_x$
and the claim is clear if $T$ is relatively free. 
So assume $T$ is arational surface as in Definition \ref{dfn_arat-surf}. Now, if $\eps$ is an edge of $S$ whose stabilizer is nonperipheral, then it is cyclic (because the only nonperipheral point stabilizers of $T$ are cyclic).
Therefore, some collapse of $S$ yields a  $\calz$-splitting of $(G,\calf)$ in which the groups $G_{Y}$ (with $Y\in\caly$) are elliptic. 
Since $T$ is arational, \cite[Proposition~11.5]{GH15-1} therefore implies that the action of each $G_{Y}$ on its minimal subtree in $T$ is simplicial. This implies that $T$ itself is simplicial, a contradiction.

Let $R$ be any Grushko tree. 
Let $Z\subseteq R$ be the minimal $G_{Y}$-invariant subtree of $R$.
Since $G_{Y}$ is nonperipheral, this is a nontrivial tree with trivial edge stabilizers.
Let $\hat{S}$ be the simplicial tree obtained by blowing up $S$ at the vertex $v_Y$ into $Z$,
and by attaching each incident edge to its unique fixed point (this is possible because all edge stabilizers are nontrivial and peripheral).

We claim that the automorphism $J_{\tilde \alpha}$ of $S$ extends to an automorphism $\hat{J}_{\tilde \alpha}$ of $\hat{S}$.
Indeed, if $I_{\tilde \alpha}$ agrees with the element $g$ on $Y$, then $J_{\tilde \alpha}$ agrees with $g$ on the vertex $v_Y$, and also on all incident edges since they are of the form $(x,Y)$ with $x\in Y$. We can then define $\hat{J}_{\tilde\alpha}$  by sending any point $x\in Z$ to $gx\in gZ$.

The splitting of $G$ obtained from $\hat{S}$ by collapsing every edge with non-trivial stabilizer yields a $H^0$-invariant free splitting of $(G,\calf)$.

For the moreover part, consider $F$ the global stabilizer of $Y$. 
Since $H_0$ is piecewise-$G$,
every $\alpha\in H_0$ has a preimage $\tilde \alpha\in\Aut(G)$ such that the restriction of $I_{\tilde \alpha}$ to $Y$ is the identity.
In particular, for every $g\in F$ and every $x\in Y$, one has $gx=I_{\tilde\alpha}(gx)=\tilde{\alpha}(g)I_{\tilde{\alpha}}(x)=\tilde{\alpha}(g)x$. It follows that for every $g\in F$, the element $g\m\tilde\alpha(g)$ fixes $Y$, so $\tilde\alpha_{|F}$ is the identity.
The proposition follows.
\end{proof}

Without any finite fix type assumption we still get the following statement which gives a converse to Proposition~\ref{stabilizers_homothetie}.

\begin{cor}\label{cor_totalement_reductible}
Let $T\in\AT$. Then $\Stabis(T)$ does not contain any fully irreducible automorphism, and no atoroidal automorphism.
\end{cor}

\begin{proof} This immediately follows from Proposition~\ref{stabilizers_isom} applied to
a cyclic subgroup $H\subseteq\Stabis(T)$ which is obviously of finite fix type because it is finitely generated.
\end{proof}

We are now in position to complete our proof of Theorem~\ref{thm_stabT}.

\begin{proof}[Proof of Theorem~\ref{thm_stabT}]
By Lemma~\ref{lem_cyclic}, either $H\subseteq\Stab^{\mathrm{is}}(T)$, or there exists $\Phi\in H$ with $\lambda_T(\Phi)\neq 1$ such that $H=H^1\rtimes\langle\Phi\rangle$ with $H^1=H\cap\Stabis(T)$. In the latter case, Proposition~\ref{stabilizers_homothetie} implies that $\Phi$ is fully irreducible, completing the proof of the first assertion of the theorem.

We now assume that $H\subseteq\Stab^{\mathrm{is}}(T)$. Corollary~\ref{cor_totalement_reductible} shows that $H$ contains no fully irreducible and no atoroidal element. If $H$ has finite fix type with respect to $T$, the last conclusion follows from Proposition~\ref{stabilizers_isom}. 
\end{proof}

\section{Classification of subgroups of $\text{Out}(G,\calf)$}\label{sec-appli}

The following theorem generalizes a theorem of Handel--Mosher \cite{HM20}, both to infinitely generated subgroups of $\text{Out}(F_N)$, and to finitely generated subgroups of relative outer automorphism groups. 
Theorem \ref{intro-alt-1} from the introduction is a particular case.

Recall that groups of finite fix type were defined in Definition \ref{dfn_fft}. See Lemmas \ref{lem_FFT_RH_FG} and \ref{lem_FFT_centralisateurs}
for the given examples of situations satisfying this condition.

\begin{theo}\label{alternative-HM}
Let $(G,\calf)$ be a nonsporadic free product. 
Let $H$ be a subgroup of $\text{Out}(G,\mathcal{F})$. Assume that $H$ has finite fix type (e.g.\ $H$ is finitely generated, or $G$ is a toral relatively hyperbolic group, or $H\subseteq\Out(G,\calf^{(t)})$ and $G$ satisfies the descending chain condition on centralizers).
\\ Then either
\begin{enumerate} 
\item[(1)] $H$ virtually preserves the conjugacy class of a proper $(G,\mathcal{F})$-free factor, or else 
\item[(2)] $H$ contains a fully irreducible automorphism. In this case, either
\begin{itemize}
\item[(2a)] $H$ contains a noncyclic free subgroup $H'$ such that every nontrivial element of $H'$ is fully irreducible, or else
\item[(2b)] $H$ contains a fully irreducible outer automorphism $\Phi$, and $H$ has a finite-index subgroup $H^0$ that splits as a semi-direct product $H^0=H^1\rtimes \langle\Phi\rangle$, where $H^1$  contains no fully irreducible element and no atoroidal element.
\end{itemize}
\end{enumerate}
\end{theo}

\begin{rk}
A version of this result has been proved by Clay and Uyanik for sporadic decompositions of the free group leading to a proof that 
atoroidal subgroups of $\Out(F_N)$ contain an atoroidal automorphism \cite{CU}.
\end{rk}

\begin{rk}\label{rk_vcyclique}
If $(G,\calf)=(F_N,\es)$ and $H$ satisfies Assertion (2b) then $H$ is virtually cyclic.
This can fail however in the general case of free products, as shown by the following example. 
Let $S$ be a compact orientable surface with one boundary component and consider the free group $G=\pi_1(S)$.
Let $c$ be an essential simple closed curve on $S$ that decomposes $S$ into two connected components $A$ and $B$, where $\partial S\subseteq A$. 
Then $\pi_1(B)$ is a proper free factor of $G$, we let $\calf=\{[\pi_1(B)]\}$. 
Let $\Phi\in\Out(G,\calf)$ be induced by a diffeomorphism of $S$ fixing $B$ and whose restriction to $A$ is pseudo-Anosov.
Then $\Phi$ is fully irreducible relative to $\calf$ and centralizes the subgroup $\Mod(B)\subset \Mod(S)$ made of all diffeomorphisms acting as the identity on $A$.
The subgroup $H=\text{Mod}(B)\times\langle\Phi\rangle$ of $\Out(G,\calf)$ which satisfies Assertion (2b) from Theorem~1.
One can also build an example where $H$ acts by a global conjugation on each subgroup in $\calf$, by taking for $H^1$ a group of twists about the curve $c$.  
\end{rk}

\begin{proof}
The group $\text{Out}(G,\calf)$ (and its subgroup $H$) acts on the free factor graph $\FF$. By a theorem of Gromov \cite[Section~8.2]{Gro87}, either $H$ contains a non-cyclic free group whose nontrivial elements are all loxodromic isometries of $\FF$, or $H$ has bounded orbits in $\FF$, or $H$ virtually fixes a point in $\partial_\infty \FF$.  In the first case, we are done because elements of $\Out(G,\calf)$ acting loxodromically on $\FF$ are fully irreducible. In the second case, Proposition~\ref{bounded-finite-ff} shows that $H$ virtually fixes the conjugacy class of a proper $(G,\calf)$-free factor. 
In the third case, $H$ virtually fixes a point $\xi\in\partial_{\infty}\FF$, and therefore it virtually preserves the finite-dimensional simplex of all trees in $\mathbb{P}\mathcal{AT}$ which project to $\xi$ (see \cite[Proposition~13.5]{GH15-1} for the fact that this set is a finite-dimensional simplex). Therefore $H$ has a finite index subgroup $H^0\subseteq H$ that fixes an extreme point $[T]$ of this simplex. The conclusion then follows from Theorem~\ref{thm_stabT}: indeed, if $H^0\subseteq\Stab^{\mathrm{is}}(T)$, then $H^0$ virtually fixes a free splitting of $(G,\calf)$, so it virtually fixes the conjugacy class of a proper $(G,\calf)$-free factor; otherwise $H^0$ is a semi-direct product as in Assertion~(2b). 
\end{proof}

We now provide a classification of subgroups of $\text{Out}(G,\calf)$ containing fully irreducible elements. This statement does not assume
that $H$ is of finite fix type.

\begin{theo}\label{alternative-hyperbolic}
Let $(G,\calf)$ be a nonsporadic free product.
Let $H\subseteq\text{Out}(G,\mathcal{F})$ be a subgroup that contains a fully irreducible outer automorphism.
\\ Then either 
\begin{itemize}
\item[(1)] $H$ virtually fixes a quadratic conjugacy class in $(G,\calf)$ (see Definition \ref{def-quadratic}), or else 
\item[(2)] $H$ contains an atoroidal fully irreducible element. In this case, either
\begin{itemize}
\item[(2a)] $H$ contains a non-cyclic free group whose non-trivial elements are all atoroidal and fully irreducible, or
\item[(2b)] $H$ has a finite-index subgroup which is a semi-direct product $H^0=H^1\rtimes\langle \Phi\rangle$
with $\Phi$ atoroidal and fully irreducible outer automorphism
and $H^1$ contains no fully irreducible element and no atoroidal element.
\end{itemize}
\end{itemize}
\end{theo}

\begin{proof}
Notice that since $H$ contains a fully irreducible automorphism, $H$ does not virtually preserve the conjugacy class of any proper $(G,\calf)$-free factor. By considering the $H$-action on $\ZF$, it follows from \cite[Section~8.2]{Gro87} that either $H$ contains a noncyclic free subgroup whose elements all act loxodromically on $\ZF$, or $H$ has bounded orbits in $\ZF$, or $H$ virtually fixes a point in $\partial_\infty \ZF$. In the first case, we are done because elements of $\Out(G,\calf)$ acting loxodromically on $\ZF$ are fully irreducible atoroidal. In the second case, Proposition~\ref{bounded-finite-zf} shows that $H$ virtually fixes a quadratic conjugacy class (recall indeed that  no finite index subgroup of $H$ preserves the conjugacy class of a proper $(G,\calf)$-free factor).   

In the third case, arguing as in the proof of Theorem~\ref{alternative-HM}, we obtain that $H$ has a finite index subgroup $H^0$ which fixes the homothety class of a relatively free arational tree. 
Let $H^1=H^0\cap\Stabis(T)$ be the isometric stabilizer of $T$ in $H^0$.
By Corollary~\ref{cor_totalement_reductible}, $H^1$  contains no fully irreducible element and no atoroidal element.
Thus $H^0\neq H^1$ (by Lemma~\ref{lem_cyclic}), and 
 $H^0=H^1\rtimes\langle \Phi\rangle$
 for some fully irreducible outer automorphism $\Phi\in H^0$.

There remains to prove that $\Phi$ is atoroidal. 
Otherwise, there would exist a nonperipheral element $g\in G$ and $k>0$ such that $\Phi^k([g])=[g]$. 
On the other hand, since the scaling factor $\lambda_T(\Phi)$ is not $1$,
 this implies that $g$ is elliptic in $T$. 
This contradicts the fact that $T$ is relatively free, and concludes the proof.
\end{proof}

Theorem \ref{intro-alt-2} from the introduction is a particular case
of the following statement.

\begin{theo}\label{thm_hyperbolic2}
Let $(G,\calf)$ be a nonsporadic free product, and let $H\subseteq\text{Out}(G,\mathcal{F})$ be a subgroup. Assume either that $H$ has finite fix type (e.g.\ $H$ is finitely generated, or $G$ is a toral relatively hyperbolic group, or $H\subseteq\Out(G,\calf^{(t)})$ and $G$ satisfies the descending chain condition on centralizers).
\\ Then either 
\begin{enumerate}
\item $H$ virtually preserves a nonperipheral conjugacy class of $G$, or the conjugacy class of a proper $(G,\calf)$-free factor, or else 
\item $H$ contains a fully irreducible atoroidal outer automorphism; in this case, either 
\begin{enumerate}
\item[(2a)] $H$ contains a nonabelian free subgroup in which all nontrivial elements are fully irreducible and atoroidal, or else
\item[(2b)] $H$ is virtually a semidirect product $H^1\rtimes\langle\Phi\rangle$, where $\Phi$ is fully irreducible and atoroidal,
and $H^1$ contains no fully irreducible element and no atoroidal element.
\end{enumerate}
\end{enumerate}
\end{theo}

\begin{proof}
Assume that $H$ does not virtually preserve the conjugacy class of any proper $(G,\calf)$-free factor. By Theorem~\ref{alternative-HM}, $H$ contains a fully irreducible element. The conclusion therefore follows from Theorem~\ref{alternative-hyperbolic}. 
\end{proof}

In the particular case of subgroups of $\text{Out}(F_N)$ containing a fully irreducible automorphism, we recover a theorem due to Uyanik \cite[Theorem 5.4]{Uya14}. The \emph{extended mapping class group} of a surface is the group of isotopy classes of diffeomorphisms of $S$ (that may reverse orientation in case $S$ is orientable) that preserve each boundary curve of the surface. 
Any identification of the fundamental group of the surface with $F_N$ gives an embedding 
of the extended mapping class group of $S$ into $\text{Out}(F_N)$.
We call such a subgroup of $\Out(F_N)$ an \emph{extended mapping class subgroup}.

\begin{theo}(Uyanik \cite[Theorem 5.4]{Uya14})\label{thm_Uyanik}
Let $H\subseteq\Out(F_N)$ be a subgroup that contains a fully irreducible automorphism. Then either $H$ contains an atoroidal fully irreducible outer automorphism, or else $H$ is contained in an extended mapping class subgroup of $\Out(F_N)$.  
\end{theo}

\begin{proof}
Assume that $H$ contains no atoroidal fully irreducible automorphism.
Theorem~\ref{alternative-hyperbolic} shows that $H$ has a finite index subgroup $H^0$ fixing a quadratic conjugacy class $c$.
 By the Dehn--Nielsen--Baer theorem (see \cite{Fuj02} for an argument in the non-orientable case),
 the stabilizer $M$ of the conjugacy class $\grp{c}$ is in an extended mapping class subgroup corresponding to a surface $S$ with a single boundary component.

We are left showing that $H$ itself embeds in $M$. 
Up to passing to a finite index subgroup of $H^0$, we can assume that $H^0$ is normal in $H$. Since $H$ contains a fully irreducible automorphism of $F_N$, so does $H^0$, which means that $H^0$ contains a pseudo-Anosov diffeomorphism of $S$. The only conjugacy class  of cyclic subgroup of $F_N$  preserved by $H^0$ is the conjugacy class of $\grp{c}$. 
It is $H$-invariant because $H^0$ is normal in $H$, so $H\subseteq M$.
\end{proof}

\bibliographystyle{alpha}
\bibliography{GH-bib}

 \begin{flushleft}
 Vincent Guirardel\\
Univ Rennes, CNRS, IRMAR - UMR 6625, F-35000 Rennes, France.\\
 \emph{e-mail: }\texttt{vincent.guirardel@univ-rennes1.fr}\\[8mm]
 \end{flushleft}

\begin{flushleft}
Camille Horbez\\
CNRS\\
Laboratoire de Math\'ematique d'Orsay, Univ. Paris-Sud, CNRS, Universit\'e Paris-Saclay,  
F-91405 ORSAY\\
\emph{e-mail: }\texttt{camille.horbez@math.u-psud.fr}
\end{flushleft}

\end{document}